\documentclass[letter,11pt]{article}

\usepackage{amsmath}

\usepackage{comment}
\usepackage[linktocpage=true]{hyperref}
\usepackage[margin=1.2 in, top=1 in, bottom= 1.2 in]{geometry}

\usepackage[english]{babel}
\usepackage{amsfonts}
\usepackage{mathrsfs}
\usepackage[mathscr]{euscript}
\usepackage{bbm}
\usepackage{latexsym}
\usepackage{math dots}
\usepackage{amssymb}
\usepackage{mathtools}
\usepackage{graphicx}
\usepackage{tikz}
\usepackage{subcaption}

\usepackage{enumitem}
\setlist{nolistsep}
\usepackage{amsthm}
\usepackage{relsize}
\usepackage{nicefrac}
\usepackage[capitalize]{cleveref}

\newtheoremstyle{plain}{3mm}{3mm}{\slshape}{}{\bfseries}{.}{.5em}{}
\newtheoremstyle{definition}{2mm}{2mm}{}{}{\bfseries}{.}{.5em}{}
\theoremstyle{plain}
	
\newtheorem{theorem}{Theorem}

\newtheorem{lemma}[theorem]{Lemma}

\theoremstyle{definition}

\theoremstyle{plain}
\newtheorem*{namedthm}{\namedthmname}
\newcounter{namedthm}
	
\newcommand{\N}{\mathbb{N}}

\newcommand{\al}{\alpha}
\newcommand{\cantor}{\overset{\infty}{\underset{l=1}{\cap}}C_l}

\title{Open intervals in sums and products of Cantor sets.}
\author{Aritro Pathak}

\begin{document}
\maketitle

\begin{abstract}
     We give new arguments for sums and products of sufficient numbers of arbitrary central Cantor sets to produce large open intervals. We further discuss the same question for $C^1$ images of such central Cantor sets. This gives another perspective on the results obtained by Astels through a different formulation on the thickness of these Cantor sets. There has been recent interest in the question of products and sums of powers of Cantor sets, and these are addressed by our methods. 
\end{abstract}

\section{Introduction}
Define the central Cantor set $C_{\al}=\cantor$ in the canonical way as the countable intersection of the sets $C_{l}$ so that the set $C_{1}$ is obtained from the unit interval by removing a middle interval $O_0$ of length $\alpha$ while keeping the intervals $A_1=[0,(1-\alpha)/2]$ and $A_2=[(1+\alpha)/2,1]$, each of measure $(1-\alpha)/2$. Further, $C_2$ is obtained from $C_1$ by deleting a middle interval of proportional length of $\alpha$, deleting in the process the intervals $O_{1}, O_{2}$ respectively from the segments $A_1, A_2$, creating four segments $A_{11}, A_{12},A_{21}, A_{22}$ in the process. This process would iterate, and at the $l$'th step, we have $2^{l}$ many segments labelled $A_{i_1\dots i_l}$ with $i_m\in \{0,1\} \ \forall \ m=1,2,\dots,k$, whose union is $C_l$. In the $l$'th step, we have removed from the segments of the $(l-1)$'th step the open middle intervals of fractional length $\alpha$, that are labelled $O_{i_1\dots i_{l-1}}$ with $i_m\in \{0,1\} \ \forall m=1,\dots,l-1$.

For $\alpha=1/3$, we get the middle third cantor set.

This paper gives an alternate dynamical method for constructing open intervals in the sums and products of a large set of Cantor sets, a question addressed in a general setting in \cite{Astels} which also followed a dynamical approximation argument. We extend the methods first presented in \cite{Pat} for just the middle third Cantor set. For the question of finding open intervals in the products of Cantor sets, that has arisen more recently in certain dynamical problems as in \cite{Takahashi}, \cite{Astels} and \cite{Takahashi} have dealt with logarithmic Cantor sets and then employed the methods for sums of Cantor sets, whereas our method presents a direct argument following \cite{Pat} that works for the product question independently of the question of Cantor sum-sets. \footnote{The methods here can also be adopted for the question of other polynomial or rational expressions of Cantor sets, although we have not pursued it here.}

For simplicity we speak momentarily of all the Cantor sets being the same but the idea is the same for the case of sums of different Cantor sets. Given a Cantor set $C_\alpha =\cantor$, \cite{Astels} works, for each $l$, with sum-sets of the disjoint segments whose union constitute $C_l$, and constructs the largest possible open interval so that each point within this open interval lies in the sum-set of these fixed number of copies of $C_l$ and this is done for each $l$. The underlying idea is essentially the following theorem in elementary analysis  which is also the approach of \cite{Athreya};

\begin{theorem}
   Suppose $\{K_i\} \subset \mathbb{R}$ are nonempty compact  sets such that $K_{1}\supset K_2 \supset K_3 \supset ..$, and $K= \cap K_{i}$. If $F: \mathbb{R}^{m} \to \mathbb{R}$ is continuous, then $F(K^{m})=\cap F(K_{i}^{m})$.
\end{theorem}

Here, $K^m$ denotes the $k$-fold Cartesian product of $K$.

Astels defines and considers the thickness of the general class of Cantor sets as defined in \cite{Astels} in order to construct the open intervals. The case of products of Cantor sets is then dealt with by first considering the logarithmic version of these Cantor sets.

Our dynamical approximation technique is a-priori conceptually different. Again for a sufficient number of copies of the same Cantor set $C_{\al}$, for a point $x$ in the open interval we start out with the points of $C_{\al}$ whose sum or product approximate $x$ to first order. These points in the first approximation all belong to the set of end points of the segments whose union constitute $C_{l_0}$ for some fixed $l_0$. In the next stage of approximation, we carefully choose endpoints of segments that constitute $C_{l_1}$ $l_{1}=l_0 +1$, which give a  controlled and better approximation of the point under consideration. We continue this for each value of $l$ so that the sums or products converge to the point $x$ in the limit as $l\to \infty$, and in the process we have Cauchy sequences of elements of the Cantor set whose limiting points sum or multiply to the point $x$. This enables us to treat the sum and product questions independent of each other.

The methods presented here are applicable to the case where we are looking for open intervals in the set of sums and products of the type:
\begin{equation*}
    \sum\limits_{i} \phi_i (C_{i}), \prod\limits_{i} \phi_i (C_{i})
\end{equation*}

where each of the functions $\phi_i$ are continuously differentiable, with upper and lower bounds on this derivative in any closed sub-interval of the support of $C_i$ and the arguments of this paper can be easily adopted for this case. For the case presented in Lemma 2 and Theorem 3, we have the special case of $\phi_i(x) =x^m$ for each $i$.\footnote{In this case one can work, instead of the upper bound of Equation 15, with some different upper bound corresponding to the bounds on the derivatives of the $\phi_i$'s.}

For the approach of Astels, one refers to \cite{Astels}. For any positive integer $m$, the case of sums of $m$'th powers of Cantor sets has attracted special attention recently, and we state our theorems for the central Cantor sets and $m$'th powers. For $m=1$, the results reduce to sums of Cantor sets. 

Finally in the Section 7, we outline the statements of results one would get for sums and products of general $C^1$ images of Cantor sets.

For more recent and earlier work on these questions, including connections to the Newhouse gap lemma and Palis' conjecture, one can refer to \cite{Athreya, Pat,Takahashi,Cantor,Cabrelli1,Cabrelli2,Yoccoz,Mult,Guo,Halmos,Majumdar,Marchese,Newhouse,Aritro,Randolph,Utz,Sum}. One notes in particular the questions raised about products of Cantor sets in the concluding section of \cite{Takahashi}. One also notes in particular the paper of \cite{Solomyak}, which gives us the sums of just two two middle-$\alpha$ Cantor sets typically has positive Lebesgue measure when the sum of the Hausdorff dimensions is greater than 1. The question of the Lebesgue measure of the sums of Cantor sets which have been addressed with Fourier analytic techniques, is not addressed with our methods.

%\Big(Also maybe explain an equivalent description as an attractor for the iterated function system; for reference check Bishop Peres's book: i.e. the attractor for the IFS consisting of two basic elements: $f_1(x)=((1-\alpha)/2)x$ and $f_2(x)=((1+\alpha)/2)+((1-\alpha)/2)x $). This is also more suitable to define the Cantor sets that have more than one deleted interval at each step. \Big)\\\\

\section{Statement of results}

Fix $\alpha \in (0,1)$. For this fixed $\alpha$ let $C_{\alpha}$ be the Cantor set as in the introduction. We define the following technical quantities:
$$s_{\alpha,m} = \left\lceil \frac{1}{\left(\frac{1+\alpha}{2}\right)^{m-1}} \right\rceil, \qquad k_{\alpha,m}=\begin{cases} \left\lceil \frac{\frac{3\alpha  - 1}{2}}{\left(\frac{1+\alpha}{2}\right)^{m-1}\left(\frac{1-\alpha^2}{4}\right)} \right\rceil & \alpha>1/3  \\
0 & \alpha \leq 1/3 \end{cases},$$
\noindent and let $r_{\alpha,m}=2s_{\alpha,m}+2k_{\alpha,m}$. Set  $$\Gamma_{\alpha, m} = \left\{\sum_{k=1}^{r_{\alpha,m}} c_k^m : c_k \in C_{\alpha}\right\}.$$

In other words, $\Gamma_{\al,m}$ contains $r_{\alpha,m}$ many $m$'th powers of elements of $C_\alpha$.
With this, we first prove in Section 2 that,

\begin{lemma} \label{thm:thm2}
Given the central Cantor set $C_\alpha$, $s_{\alpha,m},k_{\alpha,m},r_{\alpha,m}$ and $\Gamma_{\alpha,m}$ as above, let $$I = \left[\Big(\frac{r_{\alpha,m}}{2}-1\Big)+\Big(\frac{r_{\alpha,m}}{2}+1\Big)\left(\frac{1+\alpha}{2}\right)^m, \ \Big(\frac{r_{\alpha,m}}{2}+1\Big)+\Big(\frac{r_{\alpha,m}}{2}-1\Big)\left(\frac{1+\alpha}{2}\right)^m\right]$$
then, $I \subseteq \Gamma_{\alpha, m}$.
\end{lemma}

We note that for the case of $\alpha<\frac{1}{3}$, for any integer $m\geq 1$, we get with only two summands a set which is a disjoint union of intervals\footnote{In fact this would be true for any $c^2$ image of Cantor sets instead of the special case of images under the map $f(x)=x^m$}. In our theorem, for a set with a much larger number of terms, in the next theorem, we can get open intervals whose lengths are exponential in $m$.

Also consider the expression, 

\begin{align}\label{eq:ineq}
    2+\Big( \frac{1+\alpha}{2} \Big)^{m}-3\Big(\frac{1+\alpha}{2}\Big)^{m}\Big(\frac{3-\alpha}{2}\Big)^{m}.
\end{align}

When $\alpha=1$ the above quantity is zero. Further for any fixed positive integer value of $m$, one can numerically verify the existence of $0<\alpha(m)<1$ for which the above expression is $0$, and the expression changes sign when passing between the two cases of  $\alpha\leq \alpha(m)$ and $\alpha>\alpha(m)$ .
\bigskip

With an extension of the argument used to prove \cref{thm:thm2}, we also show,

\begin{theorem}
   For $\alpha\geq \alpha(m)$, there is an interval $I'$ of length,
\begin{align}
    2 + \Big(\frac{r_{\alpha,m}}{2}-3 \Big)\Big( \frac{1+\alpha}{2}\Big)^{m}\Big( \frac{3-\alpha}{2}\Big)^{m}-\Big(\frac{r_{\alpha,m}}{2}-1\Big)\Big(\frac{1+\alpha}{2}\Big)^{m},
\end{align}
with $I'\subset \Gamma_{\alpha, m}$.

For $\alpha <\alpha(m)$, we have $I'\subset \Gamma_{\alpha,m}$, a disjoint union of intervals of total length,
\begin{align}
    2r_{\alpha,m}\Big(  1- \Big(\frac{1+\alpha}{2}\Big)^{m}\Big( \frac{3-\alpha}{2}\Big)^{m} \Big).
\end{align}

Further for any $0<\alpha<1$, we have a single interval $I''\subset \Gamma_{\alpha,m}$, which is disjoint from $I'$, such that $I''$ has length 

\begin{align}
    \Big( \frac{r_{\alpha,m}}{2}+1 \Big)\Big( \frac{3+\alpha^{2}}{4} \Big)^{m} -2\Big( \frac{1+\alpha}{2} \Big)^{m} +\Big( \frac{r_{\alpha,m}}{2}-1 \Big).
\end{align}
\end{theorem}

As pointed out in Section 4 and easily verified, all of these intervals have lengths that are growing exponentially with $m$.

\bigskip

Further, we also consider in Section 5 sums of Cantor sets of different types. In this case, we construct the following main theorem. The requirement of $\alpha_1=\beta_1$ is for convenience, and such a requirement is removed later for the analogous question on the products of Cantor sets.

\begin{theorem}
    Consider any two different sets of sequences,
    \begin{align}
        \alpha_1\geq \alpha_2 \geq \dots \geq \alpha_n, \ \beta_1 \geq \beta_2 \geq \dots \beta_p, \text{with} \ \alpha_1 =\beta_1,
    \end{align}
    with integers $n,p, n_1,n_2,r_1,r_2$, that satisfy $n=n_1 +n_2, p=r_1 +r_2$, and which satisfies the following conditions:
    \begin{align}\label{ineqjm}
    \sum\limits_{j=1}^{r_1}  \Big( \frac{1+\beta_j}{2} \Big)^{m-1}\Big( \frac{1-\beta_j}{2} \Big)^{2}>\Big( \frac{3\beta_1 -1}{1+\beta_1}\Big), \ \sum\limits_{j=r_1 +1}^{r_2+r_1}  \Big( \frac{1+\beta_j}{2} \Big)^{m-1}\Big( \frac{1-\beta_j}{2} \Big)>1, 
\end{align} and, 
\begin{align}\label{ineqjm}
    \sum\limits_{j=1}^{n_1}  \Big( \frac{1+\alpha_j}{2} \Big)^{m-1}\Big( \frac{1-\alpha_j}{2} \Big)^{2}>\Big( \frac{3\alpha_1 -1}{1+\alpha_1}\Big), \ \sum\limits_{j=n_1 +1}^{n_2+n_1}  \Big( \frac{1+\alpha_j}{2} \Big)^{m-1}\Big( \frac{1-\alpha_j}{2} \Big)>1.
\end{align}

Then the interval given by:
\begin{align}
    \Bigg[ \sum\limits_{i=2}^{n}\Big(\frac{1+\alpha_i}{2}\Big)^{m}+p-1 +2\Big(\frac{1+\beta_1}{2}\Big)^{m}, \sum\limits_{i=2}^{n}\bigg(\frac{1+\alpha_i}{2}\bigg)^{m}+(p+1) \Bigg]
\end{align}
    is contained in the set, 

\begin{align}
    \Gamma = \Bigg\{ \sum\limits_{k=1}^{n} c_{i}^{m}+\sum\limits_{k=1}^{p} d_{i}^{m}: c_{i}\in C_{\alpha_i}, d_i\in C_{\beta_i} \Bigg\}
\end{align}    
 
\end{theorem}

We do not pursue this problem here, but following the methods of Theorem 3, one can expand the interval obtained in Theorem 4 in the situation of sums of distinct Cantor sets.

We introduce in Section 6, the main definitions and the statement of the Theorem 5, showing the existence of an open interval in the product of sufficient number of Cantor sets of the same type.

\section{Sums of Cantor sets.}

In this section, we prove \cref{thm:thm2}. We state this theorem for central Cantor sets, with $m=1$ being the special case of sums of arbitrary central Cantor sets.

\noindent

We seek to represent any $x\in I$ as a sum
$$x = \sum_{j=1}^{r_{\alpha, m}} c_k^m.$$
With a bit of thought, one can note that this can be achieved in the following way:
we initialize $r_{\alpha,m}$ points on $C_{\alpha}$ and let $S_1$ be the sum of their $m^{th}$ powers. These points are initialized so that $S_1$ serves as an initial approximation of $x$. In the subsequent steps, we shift these chosen points one at a time in a controlled manner and study the difference $\Delta_k = |x-S_k|$. In particular, we show that at step $k$, with the right initialization one obtains that $|\Delta_k| \leq m\left(\frac{1-\alpha}{2}\right)^{l_k}$ where $\{l_k\}$ is an increasing sequence of positive integers. This tells us that $\Delta_k \to 0$ as $k\to \infty$. With this dynamics in mind, it helps to think of each of the initialized points in $C_\alpha$ as the first value in a sequence of points all of which are contained in $C_\alpha$. Since at every step the points chosen to be a part of this sum belong to $C_\alpha$, the limit of each of these sequences is also in $C_\alpha$. The sum of the $m^{th}$ powers of these limiting values is representation of $x$ we seek.

% It is also beneficial to consider the following figures that display the relative positions and lengths of various relevant metrics that play a role in the proof below. We write $\eta_{\pm}=\frac{1\pm \alpha}{2}$.
%	\begin{figure}[h]\label{etapm}
%		\centering{
%			\resizebox{75mm}{!}{\huge\input{test2.pdf_tex}}
%			\caption{$\eta_{\pm}$ in the first iteration of $C_\alpha$}
%		}
%	\end{figure}

% In figure 2 we identify the various relevant lengths in the $k^{th}$ iteration of $C_{\alpha}.$ 
 
%\begin{figure}[h]\label{lengths}
%		\centering{
%			\resizebox{75mm}{!}{\huge\input{test.pdf_tex}}
%			\caption{$k^{th}$ iteration in $C_{\alpha}$}
%		}
%\end{figure}

\noindent With all this in place, we are ready to prove \cref{thm:thm2}.
\begin{proof}[Proof of \cref{thm:thm2}] 
We will construct $r_{\alpha,m}$ many sequences, with $a_{k}^{(j)}$ for all $k\geq 0 \text{ and } j=1,2,\dots,r_{\alpha,m}/2$, and with $ b_{k}^{(j)}$ for all  $k\geq 0, \text{ and } j=1,2,\dots,r_{\alpha,m}/2$.

\noindent We initialize these sequences for $k=1$ as follows: Let $a_{1}^{(j)} = \left(\frac{1+\alpha}{2}\right), \forall j=1,2,\dots,r_{\alpha,m}/2$ and $b_{1}^{(j)}=1, \forall j=1,2,\dots,r_{\alpha,m}/2$.
As the elements $a^{(j)}_k, b^{(j)}_k$ are defined through the algorithm described in this proof, $S_k$ is defined as the sum below:
    \begin{align}
S_k=\sum\limits_{j=1}^{r_{\alpha,m}/2}\big((a_{k}^{(j)})^{m} +(b_{k}^{(j)})^{m}\big).
    \end{align}

\noindent With our initialization, we have 
    \begin{align}
        S_1=\sum\limits_{j=1}^{r_{\alpha,m}/2} ((a_{1}^{(j)})^{m}+(b_{1}^{(j)})^{m}=\frac{r_{\alpha,m}}{2}\left(\frac{1+\alpha}{2}\right)^{m}+\frac{r_{\alpha,m}}{2}.
    \end{align}
\noindent Note that $S_1$ is the midpoint of the interval $I$. Without loss of generality consider any $x\in I$ with $x\leq S_1$. For each $k$, we define the difference:
    \begin{align}
        \Delta_{k}=x-S_{k}.
    \end{align}
Since $x\leq S_1, x\in I$, it is not hard to see that there is a unique real number $x_0 \in [\frac{1+\alpha}{2},1]$ with the property that,
\begin{align}\label{impeq}
        x= \sum\limits_{j=1}^{r_{\alpha,m}/2} \big(a_{1}^{(j)}\big)^{m} +x_{0}^{m} +\sum\limits_{j=2}^{r_{\alpha,m}/2} \big(b_{1}^{(j))}\big)^{m}
\end{align}

\noindent If $x_0 \in C_{\alpha}$, we are done. Otherwise, $x_0$ falls in the interior of an open interval $U$ cut out when constructing the sets $C_l$ whose countable intersection constitutes the set $C_\alpha$. In this case, let the length of the specific open interval in which $x_0$ falls, be $\alpha\left(\frac{1-\alpha}{2}\right)^{l_2}$. We start the dynamics by shifting $b_{1}^{(1)}$ closer to $x_0$ as follows: set $a_{2}^{(j)}=\left(\frac{1+\alpha}{2}\right), \forall j=1,2,\dots, r_{\alpha,m}/2$, and $b_{2}^{(j)}=1, \forall j=2,\dots, r_{\alpha,m}/2$, while we take $b_{2}^{(1)}$ to be the left end point of the interval $U$. It is crucial to note that in this process, for future iterations, one can decrease $b^{(1)}_2$ further by an amount $\left(\frac{1+\alpha}{2}\right)\left(\frac{1-\alpha}{2}\right)^{q}$ with $q\geq l_2 +1$, by considering $b_{3}^{(1)}$ to be the left end point of the open interval of length $\alpha\left(\frac{1-\alpha}{2}\right)^{q}$ present in the $(q-1)^{th}$ iteration of $C_{\alpha}$ that is immediately to the left of $b^{(1)}_{2}$. As will be required later, we note that we can also decrease the values $b_{2}^{(j)}=1$ for $j=2,\dots,n$ by a similar amount.\footnote{And in fact for $j=2,\dots,n$ we have leverage to decrease by bigger amounts $((1-\alpha)/2)^{q}((1+\alpha)/2)$ with $q\geq 0$ , which we won't need.} It is also worth noting that as our dynamics progresses, the sequences are set up so that all the $a^{(j)}$s can only be increased, and the $b^{(j)}$s can only be decreased by controlled amounts similar to the ones outlined above; and this in turn controls the way in which $|\Delta_k|\to 0$ as $k\to \infty$. 

\noindent From the setup above, and \cref{impeq} we get that 
\begin{align}
        \Delta_2= x_{0}^{m}-(b_{2}^{(1)})^{m}.
\end{align}

 \noindent Here $\Delta_2$ is positive, and is bounded from above by
\begin{align}\label{eqimp1}
        \left(b_{2}^{(1)}+\alpha\left(\frac{1-\alpha}{2}\right)^{l_2}\right)^{m} -(b_{2}^{(1)})^{m}= \alpha\left(\frac{1-\alpha}{2}\right)^{l_2}\cdot A_m\leq m \alpha\left(\frac{1-\alpha}{2}\right)^{l_2},
\end{align}

\noindent where $A_m$ is a sum of $m$ separate terms each of which is bounded from above by $1$.\footnote{Here we use the elementary identity $(x^{m}-y^{m})=(x-y)(x^{m-1}+x^{m-2}y+\dots+xy^{m-2}+y^{m-1})$.}

    \bigskip

%   \textbf{To remove later}: Here the proof will deviate slightly from the case where $\alpha=1/3$ that was already done. In this general case, we must deal separately with bounds of the form $((1-\alpha)/2)^{l}$ as well as $((1-\alpha)/2)^{l-1}\alpha$ and the two cases of $\alpha\lessgtr 1/3$, but this should be a minor modification.
\bigskip

\noindent Now to improve the error, we shift one of the $a_2^{(j)}$s to the right. Note that if we increase exactly one of the $a^{(j)}$ sequences\footnote{The effect is analogous if we exactly decrease one of the $b^{(j)}$ sequences in some other later step}, and take any given $l\geq l_2$, and set $a_{3}^{(1)}=a_{2}^{(1)}+ \left(\frac{1+\alpha}{2}\right)\left(\frac{1-\alpha}{2}\right)^l$, while keeping the rest unchanged, then we have $$S_3-S_2 = \left(a_{2}^{(1)}+\left(\frac{1+\alpha}{2}\right)\left(\frac{1-\alpha}{2}\right)^l\right)^{m}-(a_{2}^{(1)})^{m}=\left(\frac{1+\alpha}{2}\right)\left(\frac{1-\alpha}{2}\right)^l\cdot B_m$$

\noindent where $B_m$ consists of $m$ terms each of which is bounded from above by $1$ and bounded from below by $\left(\frac{1+\alpha}{2}\right)^{m-1}$, and thus we have 
   \begin{align}\label{eqimp2}
        m\left(\frac{1+\alpha}{2}\right)^{m-1}\left(\frac{1+\alpha}{2}\right)\left(\frac{1-\alpha}{2}\right)^{l} \leq S_3 -S_2 \leq m\left(\frac{1+\alpha}{2}\right)\left(\frac{1-\alpha}{2}\right)^{l}
   \end{align}

%   We call these two numbers,
%   \begin{align}\label{eq:delta}
%       \delta_{1,l}= m\left(\ffrac{1+\alpha}{2}\right)^{m-1}\left(\frac{1-\alpha}{2}\right)^{l}\left(\frac{1+\alpha}{2}\right) \leq \delta_{2,l} = m\left(\frac{1-\alpha}{2}\right)^{l}\left(\frac{1+\alpha}{2}\right)
%   \end{align}

\noindent At this point we are ready to begin the analysis of how these shifts improve the error. Here, we split into two cases, one for $\alpha \leq 1/3$, and one for $\alpha > 1/3$ corresponding to thinner Cantor sets. We have,
\begin{align}\label{obs}
        \alpha\gtrless \frac{1-\alpha}{2} \iff \alpha \gtrless \frac{1}{3}.
\end{align}

\noindent At the $k^{th}$ stage of the iteration, each of the sequences $a^{(j)}$ and $b^{(j)}$ are iterated so that $a_k^{(j)}$ and $b_k^{(j)}$ belong to $C_\alpha$, which in turn ensures that the sums $S_k$ converge to $x$. We illustrate below that this convergence is in a way where we are able to assert control over the differences $\Delta_k$ so that we have
\begin{align}
       |\Delta_k|=|x-S_k|\leq m\left(\frac{1-\alpha}{2}\right)^{l_{k}}, \ \ \ l_{k}\in \N \ \forall k, l_k \to \infty \ \text{as} \  k\to \infty.
\end{align}

\noindent Since all the elements $a_{k}^{(j)}, b_{k}^{(j)}$ for a fixed $j$ will belong to $C_\alpha$, in the limit as $k\to \infty$, these Cauchy sequences converge to elements $a_{\infty}^{(j)},b_{\infty}^{(j)}$ also in $C_{\alpha}$, and this gives our desired representation of $x$ as
$$x = \sum_{j=1}^{r_{\alpha, m}/2} (a_{\infty}^{(j)} +  b_{\infty}^{(j)})$$
    
\noindent We outline the argument for going from the $k=2$ to the $k=3$ step of the iteration and how in principle the dynamics repeats similarly from there on.

\subsection*{Case $\alpha \leq 1/3$ :}
\noindent In this case, from \cref{eqimp1} and \cref{obs} we get that the upper bound on $\Delta_2$ is given by:

    \begin{align}\label{eq10}
        \Delta_2\leq m\alpha\left(\frac{1-\alpha}{2}\right)^{l_2}\leq m\left(\frac{1-\alpha}{2}\right)^{l_2+1}.
    \end{align}

\noindent We will see that in this case, it will be enough to consider the $r_{\alpha,m}=2s_{\alpha,m}$ many sequences. We start by considering the unique positive integer $n_2\geq (l_2+1)$ so that:

    \begin{align}\label{eqimp3}
        m\left(\frac{1-\alpha}{2}\right)^{n_2+1}\leq \Delta_2< m\left(\frac{1-\alpha}{2}\right)^{n_2}
    \end{align}

\noindent  Given these bounds on $\Delta_2$, we want to show that increasing $a_2^{(j)}$ to $a_3^{(j)}$ enables $|\Delta_3| \leq m\left(\frac{1+\alpha}{2}\right)^{n_2+1}$. Upon showing this, we can find $n_3>n_2$ so that we have $$m\left(\frac{1-\alpha}{2}\right)^{n_3+1} \leq \Delta_3 \leq m\left(\frac{1-\alpha}{2}\right)^{n_3},$$ from where we proceed inductively. To this end, it is enough to verify the two following cases:

\subsubsection*{Case $\Delta_2 = m\left(\frac{1-\alpha}{2}\right)^{n_2+1}$ :}
  Recall that our shift sets $a^{(1)}_3=a^{(1)}_{2} + \left(\frac{1-\alpha}{2}\right)^{n_2}\left(\frac{1+\alpha}{2}\right)$, then from \cref{eqimp2} and \cref{eqimp3}, we obtain 
    \begin{equation}\label{eq12}
        \begin{split}
         \Delta_3 = \Delta_2 -(S_3 -S_2) &\geq m\left(\frac{1-\alpha}{2}\right)^{n_2 +1}-m\left(\frac{1-\alpha}{2}\right)^{n_2}\left(\frac{1+\alpha}{2}\right) \\
   &=m\left(\frac{1-\alpha}{2}\right)^{n_2}\left(\left(\frac{1-\alpha}{2}\right)-\left(\frac{1+\alpha}{2}\right)\right)\\
   &=-m\alpha\left(\frac{1-\alpha}{2}\right)^{n_2}\\ 
   &\geq -m\left(\frac{1-\alpha}{2}\right)^{n_2 +1}
        \end{split}
    \end{equation}
\noindent where we obtain the last inequality from \cref{obs}. On the other hand $\Delta_3 \leq m(\frac{1-\alpha}{2})^{n_2+1}$ is clear since $S_3-S_2 > 0$. Thus in this case, we have $$|\Delta_3| \leq m\left(\frac{1-\alpha}{2}\right)^{n_2+1}$$ as desired.

\subsubsection*{Case $\Delta_2=m\left(\frac{1-\alpha}{2}\right)^{n_2}$ :}
         Repeating the argument from the previous case, one obtains that $\Delta_3 \geq -m\left(\frac{1-\alpha}{2}\right)^{n_2+1}$. We would like to now bound $\Delta_3$ from above by $m\left(\frac{1-\alpha}{2}\right)^{n_2+1}$. To achieve this, we note from \cref{eqimp2} that shifting exactly one of the $a_{2}^{(j)}$ to $a_{3}^{(j)}$ yields $S_3-S_2 \geq m\left(\frac{1+\alpha}{2}\right)^{m}\left(\frac{1-\alpha}{2}\right)^{n_2}$. This however may not be enough to obtain a sufficient reduction in magnitude to $\Delta_2$, and so we seek to shift multiple $a_j^{(j)}$s by a similar amount. To this end, we note that upon shifting $$s_{\alpha, m} = \left\lceil \frac{1}{\left(\frac{1+\alpha}{2}\right)^{m-1}} \right\rceil$$ of the $a_{2}^{(j)}$ to $a_{3}^{(j)}$ for all $j=1,\dots, s_{\alpha,m}$ (if necessary) one obtains
        \begin{equation*}
            \begin{split}
            \Delta_3 = \Delta_2 - (S_3-S_2) &= m\left(\frac{1-\alpha}{2}\right)^{n_2} - s_{\alpha, m}m\left(\frac{1+\alpha}{2}\right)^{m}\left(\frac{1-\alpha}{2}\right)^{n_2} \\ 
            &\leq m\left(\frac{1-\alpha}{2}\right)^{n_2}\left(1-\left(\frac{1+\alpha}{2}\right)\right)\\ &= m\left(\frac{1-\alpha}{2}\right)^{n_2+1}
            \end{split}
        \end{equation*}

    %    Note the lower bound is of the type in \cref{eqimp2}, when exactly one of the sequences say $a^{(j)}_{2}$ is increased to $a^{(j)}_{3}=a^{(j)}_{2} + \big(\frac{1-\alpha}{2}\big)^{n_2}\big(\frac{1+\alpha}{2}\big)$ . 
        
     %   Thus when we choose exactly $s_{\alpha,m}=\lceil (2/(1+\alpha))^{m-1} \rceil$ many sequences and in the worst case increase each of them by $a^{(j)}_{3}=a^{(j)}_{2} + \big(\frac{1-\alpha}{2}\big)^{n_2}\big(\frac{1+\alpha}{2}\big)\ \forall j=1,2,\dots,s_{\alpha,m}$, we get that 

     %   \begin{align}
         %   \Delta_3\leq m\Big( \frac{1-\alpha}{2}\Big)^{n_2}-m\Big( \frac{1-\alpha}{2}\Big)^{n_2}\Big(\frac{1+\alpha}{2}\Big)\\
      %      =m\Big( \frac{1-\alpha}{2}\Big)^{n_2+1},
     %   \end{align}
        %and further the argument of the earlier case also ensures that $\Delta_3\geq -m\big((1-\alpha)/2\big)^{n_2 +1}$ 
        
\noindent Which also gives us $|\Delta_3| \leq m\left(\frac{1-\alpha}{2}\right)^{n_2+1}$, and the dynamics proceed inductively from here on. Depending on the implied constants and the exact location of the sequences $a^{(j)}_{k}$s, we might of course need less than the $s_{\alpha,m}$ many terms as illustrated had above, but even in the worst case, it suffices to have these many terms. \\\\

\noindent Certainly in case we have $m\left(\frac{1-\alpha}{2}\right)^{n_2+1}<\Delta_2<m\left(\frac{1-\alpha}{2}\right)^{n_2}$, which is intermediate between these two extreme cases, it is clear that we can do with increasing anywhere between $1$ and $s_{\alpha,m}$ many of the sequences $a^{(j)}$, and achieve a similar type of bound on $|\Delta_3|$. We have thus outlined exactly how the dynamical system runs in the case of going from the $k=2$ to the $k=3$, which completes our study of the case $\alpha \leq 1/3$.

\subsection*{Case $\alpha>1/3$ :}\footnote{   For the particular case of $m=1$, this constitutes the most general interesting case of arbitrary thin central Cantor sets.} In this case, we need one further modification in order to run the above dynamical system: we once again get an upper bound on $\Delta_2$ of the form
   \begin{align}\label{eqimp5}
        \Delta_2 \leq m\alpha\left(\frac{1-\alpha}{2}\right)^{l_2},
    \end{align}
    as in \cref{eqimp1}. Without loss of generality, consider that,
\begin{align}\label{eqimp5}
        m\left(\frac{1-\alpha}{2}\right)^{l_2+1}\leq \Delta_2 \leq m\alpha\left(\frac{1-\alpha}{2}\right)^{l_2},
    \end{align}

    Here we reduce the amount by which we shift our sequences, and consequently increase the number of terms that are being shifted, and $2s_{\alpha,m}$ many terms will no longer give us the desired reduction to the size of $|\Delta_k|$. 
    
    In doing this, we include a mix of `large' and 'small' shifts. We consider the number $k_{\alpha,m}$ which is the least positive integer such that 
    \begin{align}\label{eqmn}
            k_{\alpha,m}m\left(\frac{1+\alpha}{2}\right)^{m-1}\left(\frac{1-\alpha}{2}\right)^{l_2 +1} \left(\frac{1+\alpha}{2}\right) > m\alpha\left(\frac{1-\alpha}{2}\right)^{l_2} -m\left(\frac{1-\alpha}{2}\right)^{l_{2} +1}
    \end{align}
    The left hand side of the above inequality uses the lower bound of \cref{eqimp2} to quantify the lower bound on the total change to $S_2$ upon shifting $k_{\alpha, m}$ many sequences by an amount $(\frac{1-\alpha}{2})^{l_2+1}(\frac{1+\alpha}{2})$.
    
    Note that the shift is not by the amount $(\frac{1-\alpha}{2})^{l_2}(\frac{1+\alpha}{2})$ like it was in the case $\alpha \leq 1/3$, but by $(\frac{1-\alpha}{2})^{l_2+1}(\frac{1+\alpha}{2})$, and thus these $k_{\alpha,m}$ many terms correspond to 'small' shifts.
    
    Also note that,
    \begin{align}\label{eqnn}
            m\left(\frac{1-\alpha}{2}\right)^{l_2 +1}\left(\frac{1+\alpha}{2}\right)< 2m\left(\frac{1-\alpha}{2}\right)^{l_2 +1}
        \end{align}

\noindent The above inequality ensures that in the extreme case when $\Delta_2 = m\left(\frac{1-\alpha}{2}\right)^{l_2 +1}$, by increasing exactly one of the $a^{(j)}$'s for some $j\in\{ s_{\alpha,m} +1,\dots, (s_{\alpha,m}+k_{\alpha,m})\}$, by an amount of $(\frac{1+\alpha}{2})(\frac{1-\alpha}{2})^{l_2 +1}$, the decrement to $\Delta_2$ is such that $\Delta_3$ is bounded from below by $-m\left( \frac{1-\alpha}{2}\right)^{l_2 +1}$\footnote{Note that above we cannot increase the particular $a^{(j)}$ by an amount of $((1+\alpha)/2)((1-\alpha)/2)^{l_2}$ or some exponent $n\leq l_2$ under the assumption that $\alpha>1/3$. In the worst case with an increment to $a^{(j)}$ by an amount of $((1+\alpha)/2)((1-\alpha)/2)^{l_2}$, the above inequality \cref{eqnn}  reduces to $(\frac{1+\alpha}{2})<2(\frac{1-\alpha}{2}), \Leftrightarrow \alpha<1/3$.}.

        Here, \cref{eqmn} ensures that in the other extreme case when $\Delta_2 = m\Big(\frac{1-\alpha}{2}\Big)^{l_2}\alpha$, by considering up-to all of the $k_{\alpha,m}$ many $a^{(j)}$ sequences and increasing them by an amount of $(\frac{1+\alpha}{2})(\frac{1-\alpha}{2})^{l_2 +1}$, in the worst case $\Delta_3$ is bounded from above by $m\Big( \frac{1-\alpha}{2}\Big)^{l_2 +1}$.

%    This follows from noting that when exactly one $a_{3}^{(j)}$ is increased for some fixed $j$ by the amount in the previous paragraph, then we have a lower bound of the form given in \cref{eqimp2}, and thus in the worst case, when $s_{\alpha,m}=\lceil\big(\frac{2}{1+\alpha}\big)^{m-1} \rceil$ many of these sequences are increased, we get the above bound.

  %  \textbf{INCOMPLETE! : there are contingencies that might be an issue.} $\Delta_2$ may be between $m((1-\alpha)/2)^{l_2 +1}$ and $m((1-\alpha)/2)^{l_2}\alpha$ while being very close to  $m((1-\alpha)/2)^{l_2 +1}$, and the lower bound to the difference which is $m((1+\alpha)/2)^{m}((1-\alpha)/2)^{l_2 -1}$ may be too big compared to the length $2m((1-\alpha)/2)^{l_2 +1}$.

  \bigskip

    Henceforth, it is enough to verify the two following cases for any $l_3\geq l_2$:

    \begin{itemize}
        \item $\Delta_3=m\Big( \frac{1-\alpha}{2} \Big)^{l_{3}+2}$: in this case, if we increase exactly one of the sequences $a^{(j)}$ by an amount $(\frac{1-\alpha}{2})^{l_3 +1}(\frac{1+\alpha}{2})$, for some fixed $1\leq j\leq s_{\alpha,m}$, then again we would conclude that 
        \begin{align}
            \Delta_4 \geq m\Big( \frac{1-\alpha}{2} \Big)^{l_3 +2} -m\Big( \frac{1-\alpha}{2} \Big)^{l_3 +1} \Big( \frac{1+\alpha}{2}\Big)= -m\Big(\frac{1-\alpha}{2}\Big)^{l_3 +1}\alpha.
        \end{align}

        In this case, $-m\Big(\frac{1-\alpha}{2}\Big)^{l_3 +1}\alpha\leq -m\Big(\frac{1-\alpha}{2}\Big)^{l_3 +2}$, and now we again employ the argument following \cref{eqimp5}. With $k_{\alpha,m}$ many sequences $b^{(j)}$ with $b^{(j)}_{1}=1$, and which are decreased in the usual manner, for each $j \in \{s_{\alpha,m}+k_{\alpha,m}+1,\dots, s_{\alpha,m}+2k_{\alpha,m}\}$ and this would make $\Delta_4\geq -m\Big( \frac{1-\alpha}{2} \Big)^{l_3 +2}$.  Because of the symmetry of the dynamical system, we have the contingency of needing to use, in the next step, the earlier sequences for each $j\in \{s_{\alpha,m}+1,s_{\alpha,m}+k_{\alpha,m} \}$. 
        
        Note that if we increase exactly one of the $a^{(j)}$ sequences by the smaller amount of $(\frac{1-\alpha}{2})^{l_3 +2}(\frac{1+\alpha}{2})$ then we would not need the next iteration here as stated above with the further $2k_{\alpha,m}$ many terms, but then in the following step the number of terms needed changes to $\tilde{s}_{\alpha,m}=\lceil \frac{1}{((1+\alpha)/2)^{m-1} ((1-\alpha)/2)} \rceil$. In this case, it would be enough to need $2\tilde{s}_{\alpha,m}$ many terms. But this is greater than $2s_{\alpha,m}+2k_{\alpha,m}$ when $\alpha\geq 1/3$ and our number of terms is optimal. 

        \item $\Delta_3=m\Big( \frac{1-\alpha}{2} \Big)^{l_{3}+1}$: we increase up to $s_{\alpha,m}$ many of the corresponding $a^{(j)}$ and for each of them  we get a decrement to $\Delta_3$ of an amount whose magnitude is lower bounded by the left hand bound of Equation 7, which is $m(\frac{1+\alpha}{2})^{m-1}\Big((\frac{1-\alpha}{2})^{l_3+1}(\frac{1+\alpha}{2})\Big)$, for $j$ in the range $\{1,2,\dots,s_{\alpha,m}\}$, and then we get: $-m\Big(\frac{1-\alpha}{2}\Big)^{l_{3}+1}\alpha<\Delta_4 \leq m\Big(\frac{1-\alpha}{2}\Big)^{l_{3}+2}$. Note the lower bound here is the same as in the previous case. Then the dynamical system iterates again with one further use of the argument following \cref{eqimp5} and a set of upto $k_{\alpha,m}$ many further terms.
    \end{itemize}

    For the intermediate case when $-m\Big( \frac{1-\alpha}{2} \Big)^{l_{3}+1}<\Delta_3<-m\Big( \frac{1-\alpha}{2} \Big)^{l_{3}+2}$, it would be enough to decrease at least one and at most $s_{\alpha,m}$ many of the sequences $b^{(j)}$ for $1\leq j\leq s_{\alpha,m}$, and then this argument would iterate again.
    
    This completes the proof.
\end{proof}

\bigskip

\section{Expanding the open interval for sums of Cantor sets.}

Here we prove Theorem 3. 

\begin{proof}Using essentially the same dynamical argument as in the previous section, one can get an open interval to be contained inside the set $\Gamma_{\alpha,m}$, which is exponentially large in $m$.

Consider an alteration of the argument of \cref{thm:thm2} whereby for each $1\leq t\leq r_{\alpha,m}/2$, we have a dynamical system where for all $0<j\leq t$, the $a^{(j)}_1$ values are initialized to $\Big(\frac{1+\alpha}{2}\Big) +\Big(\frac{1+\alpha}{2}\Big)\Big(\frac{1-\alpha}{2}\Big)=\Big(\frac{1+\alpha}{2}\Big)\Big(\frac{3-\alpha}{2}\Big)$ while the remaining $(\frac{r_{\alpha,m}}{2} - t)$ many values of $a^{(j)}_{1}$ with $t\leq j\leq \frac{r_{\alpha,m}}{2}$ are initialized at $\Big(\frac{1+\alpha}{2}\Big)$. 

Further, in this case all the $b^{(j)}_{1}$ values are exactly $1$. In this case, the initial sum is given by 
\begin{align*}
    S^{(t)}_{1}=\Big( \frac{r_{\alpha,m}}{2} -t\Big)\Big( \frac{1+\alpha}{2} \Big)^{m} +t\Big(\frac{1+\alpha}{2}\Big)^{m}\Big( \frac{3-\alpha}{2}\Big)^{m} +\frac{r_{\alpha,m}}{2}.
\end{align*}

In this situation, we can make the dynamics of \cref{thm:thm2} work for any real number in the interval: 
\begin{align}
I^{(t)}= \Big[ \Big(\frac{r_{\alpha,m}}{2}-t\Big)\Big(\frac{1+\alpha}{2}\Big)^{m} +(t+1)\Big( \frac{1+\alpha}{2} \Big)^{m}\Big(\frac{3-\alpha}{2}\Big)^{m} +\Big( \frac{r_{\alpha,m}}{2}-1 \Big), \\  \Big(\frac{r_{\alpha,m}}{2}-t\Big)\Big(\frac{1+\alpha}{2}\Big)^{m} +(t-1)\Big( \frac{1+\alpha}{2} \Big)^{m}\Big(\frac{3-\alpha}{2}\Big)^{m} +\Big( \frac{r_{\alpha,m}}{2}+1 \Big) \Big]
\end{align}

To see this, without loss of generality, we only concentrate on one half of the above interval, which is to consider any $x\in I^{(t)}$ with $x\leq S^{(t)}_{1}$. 

In this case, for this fixed $t$ and the initializations above, consider the unique real number $x_0\in [(\frac{1+\alpha}{2})(\frac{3-\alpha}{2}),1]$ such that for $x\in I^{(t)}$,
\begin{align}\label{eq:eqimp}
     x= \sum\limits_{j=1}^{r_{\alpha,m}/2} \big(a_{1}^{(j)}\big)^{m} +x_{0}^{m} +\sum\limits_{j=2}^{r_{\alpha,m}/2} \big(b_{1}^{(j))}\big)^{m}
\end{align}

After this, for each $t$ the dynamics proceeds similarly as in the proof of \cref{thm:thm2}; one could as well have initialized all the $a^{(j)}_{1}= \Big( \frac{1+\alpha}{2} \Big)\Big(\frac{3-\alpha}{2}\Big)$ for each $1\leq j\leq r_{\alpha,m}/2$ instead of initializing $(r_{\alpha,m}/2 -t)$ many of them to take values of $(\frac{1+\alpha}{2})$ as done above, in which case the dynamics would have been concentrated in the interval $\big[\Big( \frac{1+\alpha}{2} \Big)\Big(\frac{3-\alpha}{2}\Big),1\big]$ entirely. In this case, the maximum increment permissible for all the $a^{(j)}_{1}$ terms with $1\leq j\leq t$ is $(\frac{1-\alpha}{2})^{3}(\frac{1+\alpha}{2})$, while for $t<j\leq r_{\alpha}/2$ the maximum permissible initial increment of the $a^{(j)}$ terms in this dynamics is $(\frac{1-\alpha}{2})^{2}(\frac{1+\alpha}{2})$.

\bigskip

Thus for each $1\leq t\leq r_{\alpha,m}/2$ we get an open interval $I^{(t)}$ as written above.

\bigskip

Further, for any $1\leq t \leq r_{\alpha,m}/2$, we can also do the following: for each $1\leq j \leq r_{\alpha,m}/2$ we initialize $a^{(j)}_{1}=\frac{1+\alpha}{2}$, while for each $1\leq j\leq t$ we take $b^{(j)}_{1}=\Big( \frac{1+\alpha}{2}\Big) +\Big(\frac{1-\alpha}{2}\Big)\Big(\frac{1-\alpha}{2} \Big)=\frac{3+\alpha^{2}}{4}$ and for each $t\leq j\leq r_{\alpha,m}/2$ we take $b^{(j)}_1 =1$. 

In this case, the dynamics of \cref{thm:thm2} can be made to work as in the above case, and for each $1\leq t\leq r_{\alpha,m}$, the initial sum $S_{(t),1}$ is given by
\begin{align*}
    S_{(t),1}= \frac{r_{\alpha,m}}{2}\Big(\frac{1+\alpha}{2}\Big)^{m}+t\Big(\frac{3+\alpha^{2}}{4}\Big) +\Big(\frac{r_{\alpha,m}}{2} -t\Big)
\end{align*}

Thus with the identical procedure as earlier, we get the following interval:
\begin{align}
    I_{(t)}=\Big[ \Big(\frac{r_{\alpha,m}}{2}+1 \Big)\Big(\frac{1+\alpha}{2}\Big)^{m}+(t-1)\Big(\frac{3+\alpha^{2}}{4}\Big)^{m} +\Big(\frac{r_{\alpha,m}}{2} -t\Big) , \\ \Big(\frac{r_{\alpha,m}}{2}-1 \Big)\Big(\frac{1+\alpha}{2}\Big)^{m}+(t+1)\Big(\frac{3+\alpha^{2}}{4}\Big)^{m} +\Big(\frac{r_{\alpha,m}}{2} -t\Big) \Big]
\end{align}

In this case, for any $x\in I_{(t)}$ with $x\leq S_{(t),1}$, consider the unique $x_0\in \big[\frac{1+\alpha}{2},\frac{3+\alpha^2}{4}\big]$ so that \cref{eq:eqimp} holds, and then the dynamics can be made to run as in the earlier case. In this case, the decrement initially in the $b^{(j)}$ sequence (going from $b_{1}^{(j)}\to b_{2}^{(j)}$) is given by $(\frac{1-\alpha}{2})^{l}(\frac{1+\alpha}{2})$ for some $l\geq 3$, and the maximum decrement permissible for all the $b^{(j)}_{1}$ terms with $1\leq j\leq t$ is indeed $(\frac{1-\alpha}{2})^{3}(\frac{1+\alpha}{2})$ while for $t<j\leq r_{\alpha}/2$ the maximum permissible decrement of the $b^{(j)}$ terms in the dynamics is $(\frac{1-\alpha}{2})^{2}(\frac{1+\alpha}{2})$.

\bigskip
 
Following this, depending on the values of $\alpha$ and $m$ we are able to construct a set of measure exponentially larger than the one constructed in \cref{thm:thm2}.

Note the length of each of the intervals of the form $I^{(t)}$ is $2(1-(\frac{1+\alpha}{2})^{m}(\frac{3-\alpha}{2})^{m})$, while the length of each interval of the form $I_{(t)}$ is given by $2((\frac{3+\alpha^2}{4})^{m}- (\frac{1+\alpha}{2})^{m})$.

Consider the sets

$$I'=\bigcup\limits_{i=1}^{r_{\alpha,m}/2} I^{(t)}, I''= \bigcup\limits_{i=1}^{r_{\alpha,m}/2} I_{(t)} ;$$

depending on the values of $\alpha,m$, these could be single intervals or unions of $r_{\alpha,m}/2$ many disjoint intervals. 

For $I'$, we need to check if the left end point of $I^{(t+1)}$  is less than or equal to the right end point of $I^{(t)}$; this is true when 
\begin{align}\label{eq:ineq}
    2+\Big( \frac{1+\alpha}{2} \Big)^{m}\geq 3\Big(\frac{1+\alpha}{2}\Big)^{m}\Big(\frac{3-\alpha}{2}\Big)^{m},
\end{align}

in which case $I'$ is a single interval, and if the inequality above is reversed, then $I'$ is the disjoint union of $r_{\alpha,m}/2$ many intervals. 

There is at least one case of equality in \cref{eq:ineq} for some $\alpha(m) \in (0,1)$. The typical behavior of the two components of this inequality for large $m$ is illustrated below in figure 3.

\begin{figure}
     \centering
     \begin{subfigure}[b]{0.3\textwidth}
         \centering
         \includegraphics[width=\textwidth]{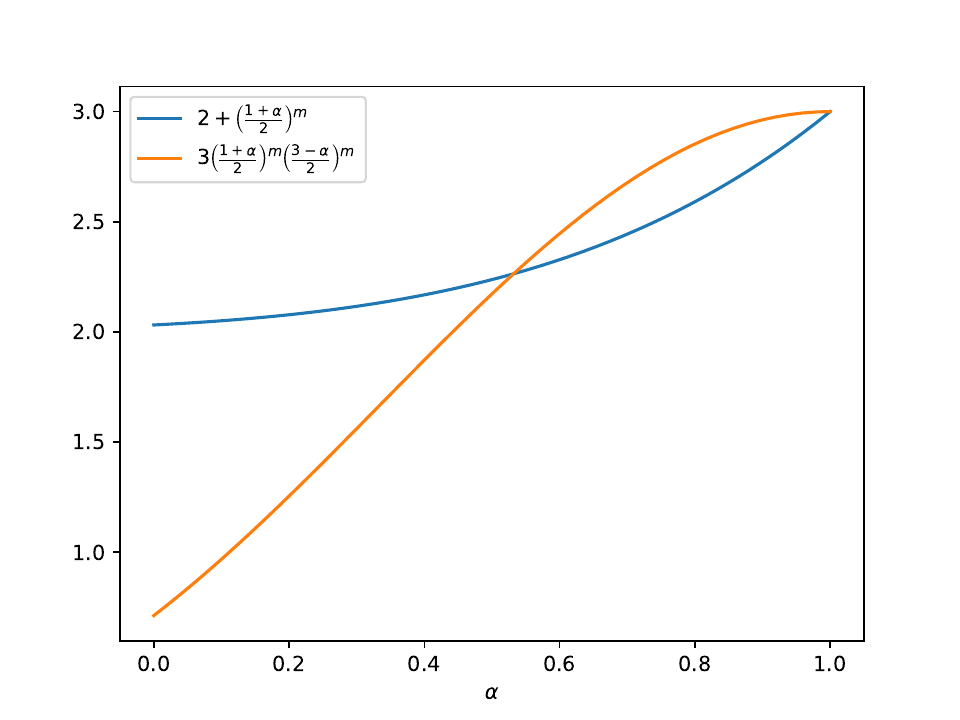}
         \caption{$m=5$}
         \label{figa}
     \end{subfigure}
     \hfill
     \begin{subfigure}[b]{0.3\textwidth}
         \centering
         \includegraphics[width=\textwidth]{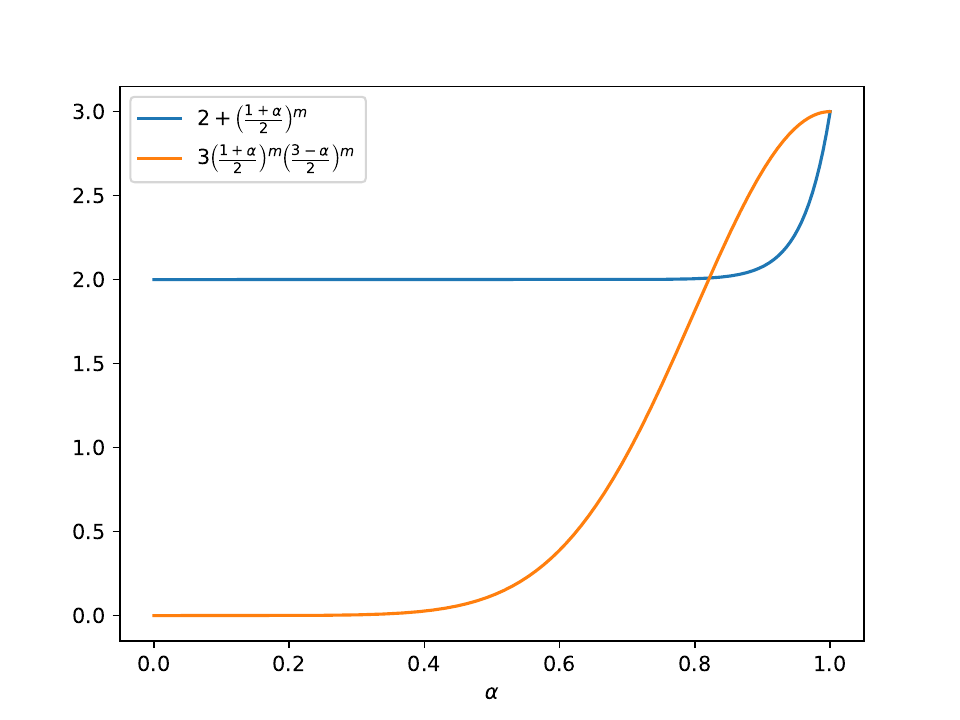}
         \caption{$m=50$}
         \label{figb}
     \end{subfigure}
     \hfill
     \begin{subfigure}[b]{0.3\textwidth}
         \centering
         \includegraphics[width=\textwidth]{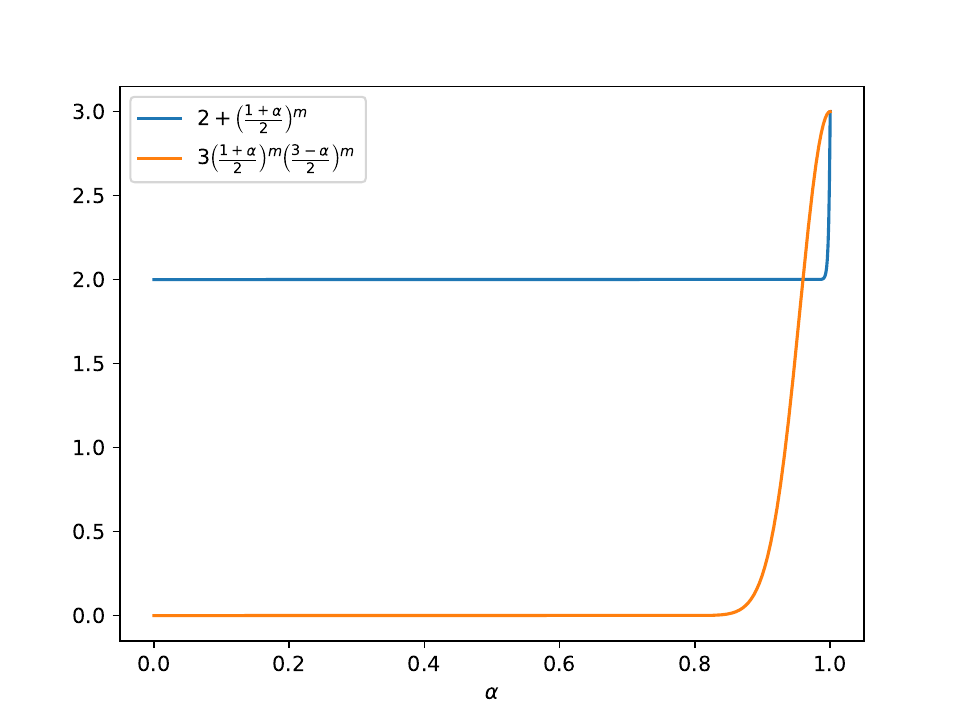}
         \caption{$m=1000$}
         \label{figc}
     \end{subfigure}
        \caption{Typical behavior of the components of inequality \ref{eq:ineq} for large $m$}
        \label{fig:three graphs}
\end{figure}

For $I''$, we need a similar check which reduces to showing: 
\begin{align}\label{eq:ineq2}
     1+ \Big( \frac{3+\alpha^{2}}{4}\Big)^{m}\geq 2\Big( \frac{1+\alpha}{2}\Big)^{m},
\end{align}

By noting that $(3+\alpha^2)\geq (1+\alpha)^2$ whenever $\alpha<1$, and Cauchy inequality, we see that the above inequality is always true.

As a consequence of this, we have that $I''$ is a always a single interval. For a given positive integer $m$, call the solution to the corresponding equality in \cref{eq:ineq} as $\alpha_1 (m)$. In this case, for $\alpha\geq \alpha(m)$, we have that $I'$ is a single interval of length 
\begin{align}
%    \Big( \frac{r_{\alpha,m}}{2}-1 \Big)\Big( \frac{1+\alpha}{2}\Big)^{m}\Big( \frac{3-\alpha}{2}\Big)^{m} +\Big( \frac{r_{\alpha,m}}{2}+1 \Big) - \Big(\frac{r_{\alpha,m}}{2}-1\Big)\Big(\frac{1+\alpha}{2}\Big)^{m} \\-2\Big( \frac{1+\alpha}{2} \Big)^{m}\Big(\frac{3-\alpha}{2}\Big)^{m}  -\Big( \frac{r_{\alpha,m}}{2}-1 \Big)=
    2 + \Big(\frac{r_{\alpha,m}}{2}-3 \Big)\Big( \frac{1+\alpha}{2}\Big)^{m}\Big( \frac{3-\alpha}{2}\Big)^{m}-\Big(\frac{r_{\alpha,m}}{2}-1\Big)\Big(\frac{1+\alpha}{2}\Big)^{m}
\end{align}

Recalling the value of $r_{\alpha,m}$ we note that the above expression grows exponentially in $m$. 

On the other hand, when $\alpha<\alpha (m)$, we have a disjoint union of intervals, and the total length of the union of these intervals is:
\begin{align}
    2r_{\alpha,m}\Big(  1- \Big(\frac{1+\alpha}{2}\Big)^{m}\Big( \frac{3-\alpha}{2}\Big)^{m} \Big),
\end{align}

and this is also exponential in $m$, since $(1+\alpha)(3-\alpha)/4<1$ and thus the expression within the brackets above is bounded from below by a suitable constant independent of $m$ and $r_{\alpha,m}$ is exponential in $m$ for any $0<\alpha< 1$.

Further, for any value of $\alpha$, $I''$ is always a single interval of length 
\begin{align}
    \Big( \frac{r_{\alpha,m}}{2}+1 \Big)\Big( \frac{3+\alpha^{2}}{4} \Big)^{m} -2\Big( \frac{1+\alpha}{2} \Big)^{m} +\Big( \frac{r_{\alpha,m}}{2}-1 \Big),
\end{align}

which is exponential in $m$.
\end{proof}

\bigskip

\section{Sums of different Cantor sets.}
Here we prove Theorem 4.

\begin{proof}In this case, we consider two different sets of sequences 
\begin{equation}\label{ine}
    \alpha_1\geq \dots\geq\alpha_n, \ \
    \beta_1\geq \dots\geq\beta_p 
\end{equation}

for some integers $n$ and $p$ to be determined later, and central Cantor sets with these parameters. With the symmetry in the dynamics, for simplicity we require first that $\alpha_1=\beta_1$, although this is not necessary.

We consider the sequences $a^{(j)}_k, k\geq 0, \forall j=1,2,\dots,n$ with $a^{(j)}_1 = (1+\alpha_j)/2$ for all $j=1,2,\dots,n$ and then $b^{(j)}_k, k\geq 0, \forall j=1,2,\dots,p$ with $b^{(j)}_1 = 1$ for all $j=1,2,\dots,p$. 

As in Section 2, consider for each $k\geq 1$:
\begin{align}
    S_{k}=\sum\limits_{j=1}^{n} (a_{k}^{(j)})^{m}+\sum\limits_{j=1}^{p}(b_{k}^{(j)})^{m}
\end{align}

Further,
   \begin{align}
        S_1=\sum\limits_{j=1}^{n} (a_{1}^{(j)})^{m}+\sum\limits_{j=1}^{p}(b_{1}^{(j)})^{m}=\sum\limits_{i=1}^{n}\Big(\frac{1+\alpha_i}{2}\Big)^{m}+p.
    \end{align}

The interval we obtain in this case is given by 
\begin{align}
    \Bigg[ \sum\limits_{i=2}^{n}\Big(\frac{1+\alpha_i}{2}\Big)^{m}+p-1 +2\Big(\frac{1+\beta_1}{2}\Big)^{m}, \sum\limits_{i=2}^{n}\bigg(\frac{1+\alpha_i}{2}\bigg)^{m}+(p+1) \Bigg]
\end{align}
    
    Consider first the case where $x\in \tilde{I}$ and $x\leq S_1$. Consider as before the unique $x_0 \in [(1+\beta_1)/2,1]$ so that 
    \begin{align}
        x=\sum\limits_{j=1}^{n} (a_{1}^{(j)})^{m}+x_{0}^{m}+\sum\limits_{j=2}^{p}(b_{1}^{(j)})^{m}
    \end{align}

    Again, if $x_0\in C_{\beta_1}$ then we are done. Otherwise as before consider the left end point of the unique open interval in whose interior $x_0$ lies and take it to be the element $b_{2}^{(1)}$, while we take $a^{(j)}_{2}=a^{(j)}_1$ for all $j=1,\dots,n$ and $b^{(j)}_2=b^{(j)}_1$ for all $j=2,\dots,p$. 

    As before, we get for some integer $l_2$, that,
    \begin{align}
        \Delta_2 =x-S_2\leq m\Big(\frac{1-\beta_1}{2}\Big)^{l_2}\beta_1
    \end{align}

    We note with some basic calculus that the expression $\Big(\frac{1-\alpha}{2}\Big)^{l} \Big(\frac{1+\alpha}{2}\Big)$ is decreasing in $\alpha$ whenever $l\geq 1$. 
    Given the chain of inequalities in \cref{ine}, we thus have:
 \begin{align}\label{section6ref}
     \Big( \frac{1-\alpha_1}{2}\Big)^{l}\Big( \frac{1+\alpha_1}{2}\Big)\leq\dots\leq \Big( \frac{1-\alpha_n}{2}\Big)^{l}\Big( \frac{1+\alpha_n}{2}\Big), \ \  \Big( \frac{1-\beta_1}{2}\Big)^{l}\Big( \frac{1+\beta_1}{2}\Big)\leq\dots\leq \Big( \frac{1-\beta_p}{2}\Big)^{l}\Big(\frac{1+\beta_p}{2}\Big).
 \end{align}   

     In this case, consider that $\alpha_1=\beta_1\geq 1/3$.

For each positive integer $l$, each $\alpha_{j}$ for $j\geq 2$ and each $\beta_{j}$ for $j\geq 2$, we choose the minimum integers $n_{\alpha_{j},l},n_{\beta_{j},l}$ such that  
\begin{align}\label{sec4ineq}
\Big(\frac{1-\alpha_{j}}{2}\Big)^{n_{\alpha_{j},l}}\Big(\frac{1+\alpha_j}{2}\Big)  \leq \Big( \frac{1-\alpha_1}{2}\Big)^{l}\Big( \frac{1+\alpha_1}{2}\Big) , \Big(\frac{1-\beta_{j}}{2}\Big)^{n_{\beta_{j},l}}\Big(\frac{1+\beta_j}{2}\Big)  \leq \Big( \frac{1-\beta_1}{2}\Big)^{l}\Big( \frac{1+\beta_1}{2}\Big)
\end{align}

%\begin{align}
%  n_{\alpha_{j},l}\log\Big( \frac{1-\alpha_j}{2} \Big) \leq  l\cdot\log\Big(\frac{1-\alpha_1}{2}\Big) -\log\Big(\frac{1+\alpha_j}{2}\Big)
%\end{align}

And thus, we get 
\begin{align}\label{eq6imp}
\begin{split}
 n_{\alpha_{j},l} =\Big\lceil \frac{1}{\log \Big(  \frac{1-\alpha_j}{2} \Big)}\Bigg(l\cdot\log\Big(\frac{1-\alpha_1}{2}\Big) +\log\Big(\frac{1+\alpha_1}{1+\alpha_j} \Big)\Bigg)\Big\rceil \\
  n_{\beta_{j},l} =\Big\lceil \frac{1}{\log \Big(  \frac{1-\beta_j}{2} \Big)}\Bigg(l\cdot\log\Big(\frac{1-\beta_1}{2}\Big) +\log\Big(\frac{1+\beta_1}{1+\beta_j} \Big)\Bigg)\Big\rceil 
  \end{split}
\end{align}

Note that the coefficient of $l$ on the right above is strictly greater than $1$ and as a result the sequence $n_{\alpha_{j},l}$ is strictly increasing in $l$. We also note that $n_{\alpha_1,l}=l$ for all positive integer $l$, and that for $j\geq 2$, we have $n_{\alpha_{j},l}$ is at least $2$.

Now we apply the techniques of Section 3 to this problem. 

We need the analogs of \cref{eqmn,eqnn} in the case. Here, \cref{eqnn} works just with $\alpha_1$. Further, for the analog of \cref{eqmn}, we need to ensure for the sequence of $\alpha_{j}$'s, for every positive integer $l$, that
\begin{align*}
    \sum\limits_{j=1}^{n_1}  m\Big( \frac{1+\alpha_j}{2} \Big)^{m-1}\Big( \frac{1-\alpha_j}{2} \Big)^{2}\Big( \frac{1-\alpha_j}{2} \Big)^{n_{\alpha_j,l}-1}\Big( \frac{1+\alpha_j}{2} \Big)>m\Big(\frac{1-\alpha_1}{2}\Big)^{l}\Big( \frac{3\alpha_1 -1}{2}\Big)
\end{align*}

For this, because of how $n_{\alpha_j ,l}$ is defined, it is enough to require that, 
\begin{align*}
    \sum\limits_{j=1}^{n_1}  \Big( \frac{1+\alpha_j}{2} \Big)^{m-1}\Big( \frac{1-\alpha_j}{2} \Big)^{2}\Big( \frac{1-\alpha_1}{2} \Big)^{l}\Big( \frac{1+\alpha_1}{2} \Big)>\Big(\frac{1-\alpha_1}{2}\Big)^{l}\Big( \frac{3\alpha_1 -1}{2}\Big),
\end{align*}

which is equivalent to,
\begin{align}\label{ineqjm}
    \sum\limits_{j=1}^{n_1}  \Big( \frac{1+\alpha_j}{2} \Big)^{m-1}\Big( \frac{1-\alpha_j}{2} \Big)^{2}>\Big( \frac{3\alpha_1 -1}{1+\alpha_1}\Big)
\end{align}

Further, for the smallest integer $n_1$ for which we have the above inequality, it is automatically ensured that we don't exceed: $m\big(\frac{1-\alpha}{2} \big)^{l}\alpha-\big(-m\big(\frac{1-\alpha_1}{2} \big)^{l+1}\big)=m\big(\frac{1-\alpha_1}{2} \big)^{l}\big(\frac{1+\alpha_1}{2}\big)$, because of \cref{sec4ineq}, \cref{eqnn} for $\alpha_1$, and because $\big(\frac{1+\alpha_j}{2}\big)^{m-1}<1$. 

Lastly, we also need the inequality corresponding to \cref{ineqs} in this case. Thus we need, for every positive integer $l$, that 
\begin{align}\label{ineqml}
    \sum\limits_{j=n_1 +1}^{n_1 +n_2}  m\Big( \frac{1+\alpha_j}{2} \Big)^{m-1}\Big( \frac{1-\alpha_j}{2} \Big)\Big( \frac{1-\alpha_j}{2} \Big)^{n_{\alpha_j,l}-1}\Big( \frac{1+\alpha_j}{2} \Big)>m\Big(\frac{1-\alpha_1}{2}\Big)^{l}\Big( \frac{1+\alpha_1}{2}\Big)
\end{align}

%\arun{Clarification: The index set should be $j=n_1+1$ to $n_1+n_2$ right?}\\\\
 
 \noindent Similar to the previous case, it is enough to require that
\begin{align}
    \sum\limits_{j=n_1 +1}^{n_2+n_1}  \Big( \frac{1+\alpha_j}{2} \Big)^{m-1}\Big( \frac{1-\alpha_j}{2} \Big)>1.
\end{align}

Further, we have $n_1 +n_2=n$ and also a similar set of constraints works for the $\beta$ parameters, with two parameters $r_1$, $r_2$ with $r_1 +r_2 =p$, and 
\begin{align}\label{ineqjm}
    \sum\limits_{j=1}^{r_1}  \Big( \frac{1+\beta_j}{2} \Big)^{m-1}\Big( \frac{1-\beta_j}{2} \Big)^{2}>\Big( \frac{3\beta_1 -1}{1+\beta_1}\Big), \ \sum\limits_{j=r_1 +1}^{r_2+r_1}  \Big( \frac{1+\beta_j}{2} \Big)^{m-1}\Big( \frac{1-\beta_j}{2} \Big)>1.
\end{align}
\end{proof}
\bigskip

We note that by switching the roles of the $\alpha_i$ and $\beta_i$ terms, and employing the same arguments and noting that $\alpha_1=\beta_1$, one can get the following interval as well:
\begin{align}
     \Bigg[ \sum\limits_{i=2}^{n}\Big(\frac{1+\beta_{i}}{2}\Big)^{m}+p-1 + 2\Big(\frac{1+\alpha_1}{2}\Big)^{m}, \sum\limits_{i=2}^{n}\big(\frac{1+\beta_{i}}{2}\big)^{m}+(p+1) \Bigg]
\end{align}

\section{Products of same Cantor set.}

Now consider the case of products of Cantor sets of the same type $C_\alpha$, in order to get an open interval, with this method. First consider the case where $\alpha\leq 1/3$.

%Given any such $0<\alpha\leq 1/3$ ,consider the quantities $t_\alpha,k_\alpha$ to be optimized with the constraint that,

%\begin{align}
%    \alpha=1-2\Bigg( 1-\Big( \frac{1}{t_\alpha}\Big)^{\frac{1}{2t_\alpha -1}} \Bigg)^{\frac{1}{k_\alpha}}.
%\end{align}

%and $t_\alpha$ and $k_\alpha$ are both taken to be as small as possible. Will determine this later. 

Consider the $k_\alpha$'th stage of the construction of the Cantor set $C_\alpha$, in which case we have a closed interval $I_{k_\alpha}$ of length $(\frac{1-\alpha}{2})^{k_{\alpha}}\in C_{k_{\alpha}}$ whose right end point is $1$. We consider the $2t_\alpha$ many sequences $a^{(j)}, b^{(j)}$, with $j=1,\dots,t_\alpha$, which are initialized such that $a^{(j)}_1=\big(1-(\frac{1-\alpha}{2})^{k_\alpha}\big)$ are all at the left end point of $I_{k_\alpha}$, and $b^{(j)}_1=1$ for each $j=1,\dots, t_\alpha$. In this case, consider the set 
\begin{align*}
    \Gamma_\alpha =\Bigg\{\prod\limits_{k=1}^{2t_\alpha} c_k: c_k \in C_\alpha, k=\overline{1,2t_\alpha} \Bigg\},
\end{align*}

where we take, the value of $t_\alpha$ to be determined later, such that,
$$\begin{cases}
    t_\alpha=s_{\alpha}  &   \alpha\leq\frac{1}{3}\\

    t_\alpha=s_{\alpha} +p_\alpha  &  \alpha>\frac{1}{3}.
\end{cases}$$

Here, $s_\alpha$ and $p_\alpha$ are determined in the equations \cref{eq:eqimpp} and \cref{eqmnn1} below, which we state here.
\begin{align}\label{eq:eqprod}
    s_\alpha=\Bigg\lceil \frac{1}{\Bigg(1-\Big(\frac{1-\alpha}{2}\Big)^{k_\alpha}\Bigg)^{2t_\alpha -1}} \Bigg\rceil, \  p_{\alpha}=\begin{cases} \Bigg\lceil \frac{(\frac{3\alpha -1}{2})}{\Big( 1- \Big(\frac{1-\alpha}{2}\Big)^{k_\alpha} \Big)^{2t_\alpha -1}(\frac{1-\alpha^2}{4})} \Bigg\rceil & \alpha>\frac{1}{3} \\ 0 & \alpha\leq \frac{1}{3} \end{cases}
\end{align}

First for the case of $\alpha\leq \frac{1}{3}$, we have
\begin{align}\label{eq:ineq6}
    \frac{1}{\Bigg(1-\Big(\frac{1-\alpha}{2}\Big)^{k_\alpha}\Bigg)^{2t_\alpha -1}}  \leq t_\alpha<1+ \frac{1}{\Bigg(1-\Big(\frac{1-\alpha}{2}\Big)^{k_\alpha}\Bigg)^{2t_\alpha -1}}
\end{align}

In this case, we consider the number $a_0>1$ such that the graph of the function $f(t)=a_{0}^{2t-1}-(t-1)$ is tangent to the $t$ axis, or in other words, the graph of the equation $y=a_{0}^{2t-1}$ is tangent to the graph of the straight line $y=t-1$. In this case, the graph of $y=a_{0}^{2t-1}$ would also cross the straight line $y=t$ and further  for any function $a^{2t-1}-(t-1)$ with $a\leq a_0$ the graph of the function $a^{2t-1}$ would cross the graph of the straight line $y=t-1$. If we choose the minimum integer $k_{\alpha}$ such that:
\begin{align}
    \frac{1}{1-\Big( \frac{1-\alpha}{2} \Big)^{k_{\alpha}}}\leq a_0 \Leftrightarrow k_{\alpha}\geq \frac{\log (1- \frac{1}{a_0})}{\log (\frac{1-\alpha}{2})},
\end{align}

which means taking,
\begin{align}\label{eq45}
    k_\alpha=\Big\lceil \frac{\log (1- \frac{1}{a_0})}{\log (\frac{1-\alpha}{2})}\Big\rceil,
\end{align}

we can find a positive integer solution to \cref{eq:ineq6} in $t_\alpha$ as well. Henceforth we consider this minimum possible solution in \cref{eq45}.

One verifies with elementary calculus that the number $a_0$ above satisfies:
\begin{align}
    a_{0}^{2 + \frac{1}{\log a_0}} =\frac{1}{2\log a_0}\Leftrightarrow a_{0}^{2}\log a_0 =\frac{1}{2e}.
\end{align}

\bigskip
Next, for the case of $\alpha> \frac{1}{3}$, with a little arithmetic combining the two parts from \cref{eq:eqprod}, it follows that:
\begin{align}\label{eq:ineq7}
    \Big(\frac{6\alpha -1 -\alpha^{2}}{1-\alpha^2}\Big)  \frac{1}{\Bigg(1-\Big(\frac{1-\alpha}{2}\Big)^{k_\alpha}\Bigg)^{2t_\alpha -1}}\leq t_\alpha< 2+ \Big(\frac{6\alpha -1 -\alpha^{2}}{1-\alpha^2}\Big)\frac{1}{\Bigg(1-\Big(\frac{1-\alpha}{2}\Big)^{k_\alpha}\Bigg)^{2t_\alpha -1}}
\end{align}

We call: 
\begin{align}
\beta_{\alpha}=\frac{1-\alpha^2}{6\alpha -1 -\alpha^{2}}
\end{align}

In this case, the above reduces to:
\begin{align}
    \beta_\alpha (t_\alpha -2)< \frac{1}{\Big( 1-\big( \frac{1-\alpha}{2}\big)^{k_\alpha} \Big)^{2t_\alpha -1}}\leq \beta_\alpha t_\alpha.
\end{align}

Here again, consider the number $a_1$ such that the graph of the function $g_{\alpha}(t)= a_{1}^{2t-1} -\beta_{\alpha}(t-2)$ is tangent to the $t$ axis, or that the graph of the equation $y=a_{1}^{2t-1}$ is tangent to the graph of the straight line $y=\beta_{\alpha}(t-3)$. By similar arguments as before, we are guaranteed to have positive integer solution in $t_\alpha$ when $k_\alpha$ is chosen to be the smallest integer such that  
\begin{align}
    \frac{1}{1-\Big( \frac{1-\alpha}{2} \Big)^{k_{\alpha}}}\leq a_1 \Leftrightarrow k_{\alpha}\geq \frac{\log (1- \frac{1}{a_1})}{\log (\frac{1-\alpha}{2})},
\end{align}

and thus choosing 
\begin{align}
    k_{\alpha}=\Big\lceil \frac{\log (1- \frac{1}{a_1})}{\log (\frac{1-\alpha}{2})}\Big\rceil,
\end{align}

we can get at least one positive integer solution $t_\alpha$ to \cref{eq:ineq7} as well. Henceforth we choose the minimum possible solution.

Again with elementary calculus one verifies that the number $a_1$ satisfies:
\begin{align}
    a_1^{5+\frac{1}{\log a_1}}=\frac{\beta_{\alpha}}{2\log a_1}\Leftrightarrow a_{1}^{5}\log a_1=\frac{\beta_\alpha}{2e}
\end{align}

\bigskip

Once we get a positive integer solution $t_{\alpha}$ to \cref{eq:ineq7}, we can also find integers $s_\alpha, p_\alpha$ that satisfy \cref{eq:eqprod} with the condition that $t_\alpha=s_\alpha + p_\alpha$. From the geometry of the line $y=\beta_{\alpha}t$ and the graph of $y=a_{1}^{2t-1}$ one sees that the solution $t_\alpha$ is at least $1/\beta_{\alpha}$ which is the point where the line $y=\beta_{\alpha}t$ intersects the line $y=1$.

Also consider the interval,
\begin{align}
    I_\alpha=\Bigg[\Bigg(1-\Big(\frac{1-\alpha}{2}\Big)^{k_\alpha}\Bigg)^{t_\alpha+1}, \Bigg(1-\Big(\frac{1-\alpha}{2}\Big)^{k_\alpha}\Bigg)^{t_\alpha -1}\Bigg]
\end{align}

\begin{theorem}\label{thm:thm3}
    Given the central Cantor set $C_\alpha$ and $t_\alpha,k_\alpha, I_\alpha,\Gamma_\alpha$ as above, we have:
    
    $$I_\alpha \subset \Gamma_\alpha.$$ 
\end{theorem}

\begin{proof}
    Without loss of generality, we only show that any $x\in \Bigg[\Bigg(1-\Big(\frac{1-\alpha}{2}\Big)^{k_\alpha}\Bigg)^{t_\alpha+1}, \Bigg(1-\Big(\frac{1-\alpha}{2}\Big)^{k_\alpha}\Bigg)^{t_\alpha}\Bigg]$ belongs to $\Gamma_\alpha$.

        Consider for each $j\geq 1$, the expression,
    \begin{align}
        P_j=\prod\limits_{k=1}^{t_\alpha}a^{(k)}_j b^{(k)}_j.
    \end{align}

    Also, define for each $j$, 

    \begin{align}
        \Delta_j =x-P_j.
    \end{align}

    In this case, consider the unique real number $x_0$ such that $x=a^{(1)}_1\dots a^{(t_\alpha)}_1 x_0 b^{(2)}_1\dots  b^{(t_\alpha)}_1$. If $x_0\in C_\alpha$ we are done, otherwise, we consider the unique open interval in the further construction of $C_\alpha$ restricted to $I_{k_\alpha}$, in which $x_0$ lies, and take $b^{(1)}_2$ to be the left end point of this interval, and keep $a^{(j)}_2 =a^{(j)}_1 \forall j=1,\dots,t_\alpha$, and $b^{(j)}_2 =b^{(j)}_1 \forall j=2,\dots,t_\alpha$. Let this open interval have a length $\alpha\big( \frac{1-\alpha}{2} \big)^{k_1}$, with clearly $k_1\geq k_\alpha$.

    In this case, we have 

    \begin{align}
        \Delta_2=x-P_2= (x_0 - b_{2}^{(1)})\prod\limits_{k=1}^{t_\alpha}a^{(k)}_2 \prod\limits_{k=2}^{t_\alpha} b^{(k)}_2 <\alpha\Big( \frac{1-\alpha}{2} \Big)^{k_1}\Bigg(1-\Big(\frac{1-\alpha}{2}\Big)^{k_\alpha}\Bigg)^{t_\alpha}\cdot 1\\ <\alpha\Big( \frac{1-\alpha}{2} \Big)^{k_1} \\ <\Big( \frac{1-\alpha}{2} \Big)^{k_1 +1},
    \end{align}

    in this case, since $\alpha\leq 1/3$, we have the last inequality. 

    Further consider the unique $k_2 \geq k_1 +1$ such that 

    \begin{align}
        \Big( \frac{1-\alpha}{2} \Big)^{k_2 +1}\leq \Delta_2 < \Big( \frac{1-\alpha}{2} \Big)^{k_2}
    \end{align}

    In this case, we would like to increase one or more of the $a^{(j)}$ sequences so that $\Delta_{2}$ decreases to a value of $\Delta_3$, with $|\Delta_3|\leq \big( \frac{1-\alpha}{2} \big)^{k_3}$ with $k_3> k_2$.

    If we increase just $a^{(1)}_{2}$ to $$a^{(1)}_{3}=a^{(1)}_{2} +\Big( \frac{1+\alpha}{2}\Big)\Big( \frac{1-\alpha}{2}\Big)^{k_2},$$

    and keep $a^{(j)}_3=a^{(j)}_2$ for all $j=2,\dots,m$, and $b^{(j)}_3=a^{(j)}_2$ for all $j=1,\dots,m$, we have,

    \begin{align}\label{eq62}
    \begin{split}
      \Bigg(1-\Big(\frac{1-\alpha}{2}\Big)^{k_\alpha}\Bigg)^{2t_\alpha -1} \Big( \frac{1+\alpha}{2}\Big)\Big( \frac{1-\alpha}{2}\Big)^{k_2} < P_3 -P_2= (a^{(1)}_3 -a^{(1)}_2)\prod\limits_{k=2}^{t_\alpha}a^{(k)}_2 \prod\limits_{k=1}^{t_\alpha} b^{(k)}_2 \\ < \Big( \frac{1+\alpha}{2}\Big)\Big( \frac{1-\alpha}{2}\Big)^{k_2}
      \end{split}
    \end{align}

    We call, 
    \begin{align}
        \theta_\alpha=\Bigg(1-\Big(\frac{1-\alpha}{2}\Big)^{k_\alpha}\Bigg).
    \end{align}

    As in the case of \cref{thm:thm2}, it is now enough to consider the following two cases:
    \begin{itemize}
        \item $\Delta_2=\Big( \frac{1-\alpha}{2} \Big)^{k_2 +1}$. In this case, we have 
        
        \begin{align}
            \Delta_3=\Delta_2 -(P_3 -P_2)> \Big( \frac{1-\alpha}{2} \Big)^{k_2 +1} - \Big( \frac{1+\alpha}{2}\Big)\Big( \frac{1-\alpha}{2}\Big)^{k_2}\\ > -\Big(  \frac{1-\alpha}{2}\Big)^{k_2}\alpha \\ > - \Big(  \frac{1-\alpha}{2}\Big)^{k_2+1}
        \end{align}

        \item $\Delta_2=\Big( \frac{1-\alpha}{2} \Big)^{k_2}$. In this case, in the worst case, by considering $t_\alpha=\lceil \frac{1}{\theta_\alpha^{2t_\alpha -1} }\rceil$ many sequences $a^{(j)}$ which we increase each by an amount $\Big( \frac{1+\alpha}{2}\Big)\Big( \frac{1-\alpha}{2}\Big)^{k_2}$, we get, in a case analogous to \cref{thm:thm2} by considering the lower bound on $P_3 -P_2$, that $\Delta_3 <\big(\frac{1-\alpha}{2} \big)^{k_2+1}$.
    \end{itemize}

    \bigskip

    Thus we have the condition for $\alpha\leq 1/3$:
    \begin{align}\label{eq:eqimpp}
        t_\alpha=s_\alpha=\Bigg\lceil \frac{1}{\theta_\alpha} \Bigg\rceil=\Bigg\lceil \frac{1}{\Bigg(1-\Big(\frac{1-\alpha}{2}\Big)^{k_\alpha}\Bigg)^{2t_\alpha -1}} \Bigg\rceil
    \end{align}

    Further on, the dynamics works in the same way so that we get $|\Delta_k| =|x-P_k|\to 0$ as $k\to \infty$.

    \bigskip

    Next, the case when $\alpha>1/3$ is again very similar to the corresponding case of \cref{thm:thm2}: the analogue of \cref{eqnn}is given explicitly as:
    \begin{align}\label{eqnnn}
        \Big( \frac{1+\alpha}{2}\Big)\Big( \frac{1-\alpha}{2}\Big)^{k_2 +1}< 2\Big( \frac{1-\alpha}{2} \Big)^{k_2 +1}
    \end{align}

    Also, the case corresponding to \cref{eqmn} is given by considering the least positive integer $p_{\alpha}$ so that:
    \begin{align}\label{eqmnn}
    p_{\alpha}\theta_{\alpha}^{2t_\alpha -1} \Big( \frac{1+\alpha}{2}\Big)\Big( \frac{1-\alpha}{2}\Big)^{k_2+1}>\Big(\frac{1-\alpha}{2}\Big)^{k_2 }\Big(\frac{3\alpha-1}{2}\Big)\\ =\Big(\frac{1-\alpha}{2}\Big)^{k_2 }\alpha -\Big(\frac{1-\alpha}{2}\Big)^{k_{2}+1}
    \end{align}

    This gives us,
    \begin{align}\label{eqmnn1}
        p_{\alpha}=\Bigg\lceil \frac{(\frac{3\alpha -1}{2})}{\Big( 1- \Big(\frac{1-\alpha}{2}\Big)^{k_\alpha} \Big)^{2t_\alpha -1}(\frac{1-\alpha^2}{4})} \Bigg\rceil,
    \end{align}

    while we also have 
    \begin{align}\label{eqmnn2}
    s_{\alpha}=\Bigg\lceil \frac{1}{\Big( 1- \Big(\frac{1-\alpha}{2}\Big)^{k_\alpha} \Big)^{2t_\alpha -1}}\Bigg\rceil,
    \end{align}

    and we recall that $t_\alpha=s_\alpha +p_\alpha$.

    Thus we have that:
    \begin{align}
        t_\alpha = \Bigg\lceil \frac{1}{\Big( 1- \Big(\frac{1-\alpha}{2}\Big)^{k_\alpha} \Big)^{2t_\alpha -1}}\Bigg\rceil + \Bigg\lceil \frac{(\frac{3\alpha -1}{2})}{\Big( 1- \Big(\frac{1-\alpha}{2}\Big)^{k_\alpha} \Big)^{2t_\alpha -1}(\frac{1-\alpha^2}{4})} \Bigg\rceil.
    \end{align}

        The condition in \cref{eqnnn} above ensures that in the case when $\Delta_2 = \Big(\frac{1-\alpha}{2}\Big)^{k_2+1}$, by increasing exactly one of the $a^{(j)}$'s, for some $j\in\{1+ t_{\alpha} ,\dots, (p_{\alpha}+t_{\alpha})\}$, by an amount of $(\frac{1+\alpha}{2})(\frac{1-\alpha}{2})^{k_2 +1}$, that in the worst case, the decrement to $\Delta_2$ is such that $\Delta_3$ is bounded from below by $-\Big( \frac{1-\alpha}{2}\Big)^{k_2 +1}$.

        The second case above ensures that in the other extreme case when $\Delta_2 = \Big(\frac{1-\alpha}{2}\Big)^{k_2}\alpha$, by considering up-to all of the $p_{\alpha}$ many $a^{(j)}$ sequences and increasing them by an amount of $((1+\alpha)/2)((1-\alpha)/2)^{k_2 +1}$, that in the worst case, we have that $\Delta_3$ is bounded from above by $\Big( \frac{1-\alpha}{2}\Big)^{k_2 +1}$.

        This process iterates at every step, with the argument repeating for each of the stages with the exponents $k_q \to \infty$ as $q\to \infty$.

        Further on, the dynamics iterates exactly as in the case of $\alpha>1/3$ for \cref{thm:thm2}; we would need a further $2p_{\alpha}$ many sequences.

        For some $k_3\geq k_2$, it is enough to consider the following two cases:

         \begin{itemize}
        \item $\Delta_3=\Big( \frac{1-\alpha}{2} \Big)^{k_{3}+2}$; in this case, if we increase exactly one of the sequences $a^{(j)}$ for some fixed $1\leq j\leq s_{\alpha}$, then again we would conclude that 
        \begin{align}
            \Delta_4 \geq -\Big(\frac{1-\alpha}{2}\Big)^{k_3 +1}\alpha.
        \end{align}

        In this case, $\Big(\frac{1-\alpha}{2}\Big)^{k_3 +1}\alpha\geq \Big(\frac{1-\alpha}{2}\Big)^{k_3 +2}$, and we would have to recursively employ the argument following \cref{eqmnn1} above, with $p_{\alpha}$ many sequences $b^{(j)}$ with $b^{(j)}$ being initialized to $1$, and which are decreased in the manner identical to the one above, and this would make $\Delta_4\geq -\Big( \frac{1-\alpha}{2} \Big)^{k_3 +2}$. Because of the symmetry of the dynamical system, we have the requirement of the other set of $p_\alpha$ sequences in the later step. 

        \item $\Delta_3=\Big( \frac{1-\alpha}{2} \Big)^{k_{3}+1}$; here we have to increase up to $s_{\alpha}$ many of the corresponding $a^{(j)}$ and for each of them  we get a decrement to $\Delta_3$ of the amount whose magnitude is lower bounded by $\Big(1-\Big(\frac{1-\alpha}{2}\Big)^{k_\alpha}\Big)^{2t_\alpha -1} \Big( \frac{1+\alpha}{2}\Big)\Big( \frac{1-\alpha}{2}\Big)^{k_2} $, for $j$ in the range $\{1,2,\dots,s_{\alpha}\}$, and then again we get: $$-\Big(\frac{1-\alpha}{2}\Big)^{k_{3}+1}\alpha<\Delta_4 \leq \Big(\frac{1-\alpha}{2}\Big)^{k_{3}+2},$$ and the dynamical system iterates again with the use possibly of $2p_\alpha$ many sequences as in the previous case.
    \end{itemize}

\end{proof}

Given any fixed $0<\alpha<1$, we have taken the smallest permissible positive integer $t_\alpha$ and $k_\alpha$, in order to get the fewest number of terms needed to get the biggest possible open interval.

\section{Products of different Cantor sets.}

In this case, as in Section 5, we consider two different sets of sequences;
\begin{equation}\label{ine}
    \alpha_1\geq \dots\geq\alpha_n, \ \
    \beta_1\geq \dots\geq\beta_p,
\end{equation}

for some integers $n,p$ to be determined later, and central Cantor sets with these parameters. We also require for simplicity that $\alpha_1 =\beta_1$, although this is not strictly necessary. 

As in Section 5, we consider the sequences $a^{(j)}_k, k\geq 0, \forall j=1,2,\dots,n$ with $a^{(j)}_1 = (1+\alpha_j)/2$ for all $j=1,2,\dots,n$ and then $b^{(j)}_k, k\geq 0, \forall j=1,2,\dots,p$ with $b^{(j)}_1 = 1$ for all $j=1,2,\dots,p$. Consider similar to Section 5, for each $k\geq 1$, the expression 
    \begin{align}
        P_k=\prod\limits_{j=1}^{n}a^{(j)}_k \prod\limits_{j=1}^{p} b^{(j)}_k.
    \end{align}

    Recall, we defined in Section 6,
        \begin{align*}
        \theta_\alpha=\Bigg(1-\Big(\frac{1-\alpha}{2}\Big)^{k_\alpha}\Bigg).
    \end{align*}

and here we define a parameter, 

    \begin{align}\label{eq77}
        \chi:=\Bigg(\prod\limits_{i=1}^{n}\theta_{\alpha_i}\prod\limits_{i=1}^{p}\theta_{\beta_i}\Bigg)
    \end{align}

We note that in this case we are making the lower bound corresponding to \cref{eq62} weaker by putting in an additional $\theta_{\alpha_1}$ factor on the left. Thus the contribution to the total number of terms, as in \cref{eq:eqimpp} increases in this case.

We state the main theorem:
\begin{theorem}
    Consider the sequences $\alpha_1\geq \alpha_2\geq \dots\geq \alpha_n$ and $\beta_1\geq \beta_2\geq\dots\beta_p$ as defined earlier, with $\alpha_1 =\beta_1$, with the constraints that,
\begin{align}\label{eq79}
    \chi\sum\limits_{i=1}^{n_1}\Big( \frac{1-\alpha_i}{2} \Big)^{2}>\Big( \frac{3\alpha_1 -1}{1+\alpha_1} \Big),  \    \chi \sum\limits_{i=1+n_1}^{n_1 + n_2} \Big( \frac{1-\alpha_i}{2} \Big)>1.
\end{align}    
     Here, $\chi$ is defined above, and we have $n_1 +n_2=n$. Further, we require, with $r_1 +r_2 =p$, that,
\begin{align}
    \chi\sum\limits_{i=1}^{r_1}\Big( \frac{1-\beta_i}{2} \Big)^{2}>\Big( \frac{3\beta_1 -1}{1+\beta_1} \Big),\      \chi \sum\limits_{i=1+r_1}^{r_1 + r_2} \Big( \frac{1-\beta_i}{2} \Big)>1.
\end{align}       
    
    Then the following intervals, 
\begin{align}
    \Bigg[ \Bigg( \prod\limits_{i=1}^{p} \Bigg( 1-\bigg(\frac{1-\beta_i}{2}\bigg)^{k_{\beta_i}} \Bigg) \Bigg) \Bigg(  1-\bigg(\frac{1-\alpha_1}{2}\bigg)^{k_{\alpha_1}} \Bigg) , \Bigg( \prod\limits_{i=2}^{p} \Bigg( 1-\bigg(\frac{1-\beta_i}{2}\bigg)^{k_{\beta_i}} \Bigg) \Bigg)\Bigg],
\end{align}

\begin{align}
    \Bigg[ \Bigg( \prod\limits_{i=1}^{n} \Bigg( 1-\bigg(\frac{1-\alpha_i}{2}\bigg)^{k_{\alpha_i}} \Bigg) \Bigg) \Bigg(  1-\bigg(\frac{1-\beta_1}{2}\bigg)^{k_{\beta_1}} \Bigg) , \Bigg( \prod\limits_{i=2}^{n} \Bigg( 1-\bigg(\frac{1-\alpha_i}{2}\bigg)^{k_{\alpha_i}} \Bigg) \Bigg)\Bigg],
\end{align}
    are contained in the set, 
    \begin{align}
        \Gamma =\Bigg\{ \prod\limits_{k=1}^{n}c_k\prod\limits_{l=1}^{p} d_l: c_k \in C_{\alpha_k}, d_{l}\in C_{\beta_l}  \Bigg\}
    \end{align}

\end{theorem}

    In this case, for each $l\geq 1$, we again have the chain of inequalities from \cref{section6ref}, and we further define for each positive integer $l$, for each $\alpha_j\geq 2$ and for each $\beta_j \geq 2$, we find the minimum integers $n_{\alpha_j,l}$ and $n_{\alpha_j,l}$ so that,
\begin{align}\label{sec4ineq}
\Big(\frac{1-\alpha_{j}}{2}\Big)^{n_{\alpha_{j},l}}\Big(\frac{1+\alpha_j}{2}\Big)  \leq \Big( \frac{1-\alpha_1}{2}\Big)^{l}\Big( \frac{1+\alpha_1}{2}\Big) , \Big(\frac{1-\beta_{j}}{2}\Big)^{n_{\beta_{j},l}}\Big(\frac{1+\beta_j}{2}\Big)  \leq \Big( \frac{1-\beta_1}{2}\Big)^{l}\Big( \frac{1+\beta_1}{2}\Big).
\end{align}

This gives us the expressions for $n_{\alpha_j,l},n_{\beta_j,l}$ as in \cref{eq6imp}. Now we use the techniques of Section 5 in this problem. Without loss of generality, we work with the case where $\alpha_1=\beta_1 \geq 1/3$. We work with the analog of \cref{eqmnn} with $\alpha_1$, and require:
\begin{align*}
    \chi\sum\limits_{i=1}^{n}\Big( \frac{1-\alpha_i}{2} \Big)^{2}\Big( \frac{1-\alpha_i}{2} \Big)^{n_{\alpha_i,l}-1}\Big( \frac{1+\alpha_i}{2} \Big)>\Big( \frac{1-\alpha_1}{2} \Big)^{l}\Big( \frac{3\alpha_1 -1}{2} \Big).
\end{align*}

Because of how the $n_{\alpha_i,l}$ are defined, it is enough to require that,
\begin{align*}
    \chi\sum\limits_{i=1}^{n_1}\Big( \frac{1-\alpha_i}{2} \Big)^{2}\Big( \frac{1-\alpha_1}{2} \Big)^{l}\Big( \frac{1+\alpha_1}{2} \Big)>\Big( \frac{1-\alpha_1}{2} \Big)^{l}\Big( \frac{3\alpha_1 -1}{2} \Big).
\end{align*}

This gives us the requirement,
\begin{align}\label{eq79}
    \chi\sum\limits_{i=1}^{n_1}\Big( \frac{1-\alpha_i}{2} \Big)^{2}>\Big( \frac{3\alpha_1 -1}{1+\alpha_1} \Big).
\end{align}

We further require an analog of \cref{eq:eqimpp}:
\begin{align*}
    \chi \sum\limits_{i=1}^{n_1} \Big( \frac{1-\alpha_i}{2} \Big) \Big( \frac{1+\alpha_i}{2} \Big)\Big( \frac{1-\alpha_i}{2} \Big)^{n_{\alpha_i,l}-1} >\Big( \frac{1-\alpha_1}{2} \Big)^{l}\Big(\frac{1+\alpha_1}{2}\Big).
\end{align*}

Further, because of the definition of $n_{\alpha_i,l}$, it is enough to require that:
\begin{align}\label{eq80}
    \chi \sum\limits_{i=1+n_1}^{n_1 + n_2} \Big( \frac{1-\alpha_i}{2} \Big)>1.
\end{align}

and finally we have $n_1 +n_2=n$. An analogous set of constraints is needed for the $\beta_i$ parameters. Given \cref{eq77}, the \cref{eq79,eq80} are coupled inequalities in the parameters $n_1$ and $n_2$. This would have a set of solutions in the parameters $\alpha_i$ and then the $n_1,n_2$. A similar statement holds for $\beta_i$ parameters.

\bigskip

We end with a discussion for the arbitrary sequence $\gamma_n\leq \gamma_{n-1}\leq \dots\leq \gamma_1$. We can split it up into two disjoint subsequences $\alpha_k\leq \alpha_{k-1}\leq \alpha_{k-2}\leq \alpha_1 $ and $\beta_j \leq \beta_{j-1}\leq \dots\leq \beta_1$ and consider the sets $A=\{\alpha_t:\alpha_t\leq \beta_1\}$ and $B=\{\beta_t:\beta_t\leq \alpha_1\}$. In this case, one can  verify that it would be enough to require analogs of \cref{eq79,eq80} which are the following:
\begin{align}\label{eq81}
    \chi_1\sum\limits_{i\in A_1}\Big( \frac{1-\alpha_i}{2} \Big)^{2}>\Big( \frac{3\beta_1 -1}{1+\beta_1} \Big),\ \ \chi_1 \sum\limits_{i\in A_2} \Big( \frac{1-\alpha_i}{2} \Big)>1.
\end{align}

Further we also require:
\begin{align}\label{eq82}
    \chi_1\sum\limits_{i=1}^{n_1}\Big( \frac{1-\beta_i}{2} \Big)^{2}>\Big( \frac{3\beta_1 -1}{1+\beta_1} \Big),\ \ \chi_1 \sum\limits_{i=(n_1 +1)}^{n_1 +n_2} \Big( \frac{1-\beta_i}{2} \Big)>1.
\end{align}

Here we have the disjoint union $A=A_1 \sqcup A_2$ and also $n_1 +n_2 =j$.

Symmetrically, we can also have the following:
\begin{align}\label{eq83}
    \chi_2\sum\limits_{i\in B_1}\Big( \frac{1-\beta_i}{2} \Big)^{2}>\Big( \frac{3\alpha_1 -1}{1+\alpha_1} \Big),\ \ \chi_2 \sum\limits_{i\in B_2} \Big( \frac{1-\beta_i}{2} \Big)>1,
\end{align}

as well as,
\begin{align}\label{eq84}
    \chi_2\sum\limits_{i=1}^{m_1}\Big( \frac{1-\alpha_i}{2} \Big)^{2}>\Big( \frac{3\alpha_1 -1}{1+\alpha_1} \Big),\ \ \chi_2 \sum\limits_{i=(m_1 +1)}^{m_1 +m_2} \Big( \frac{1-\alpha_i}{2} \Big)>1,
\end{align}

where $B=B_1 \sqcup B_2$ and also $m_1 +m_2=k$. 

Here, $\chi_1:=\prod\limits_{i\in A}\theta_{\alpha_i}\prod\limits_{t=1}^{j} \theta_{\beta_t},\ \ \chi_2:=\prod\limits_{i\in B}\theta_{\beta_i}\prod\limits_{t=1}^{k} \theta_{\alpha_t}$.  

Associated to each $\gamma_i, i\in \{1,\dots,n\}$ is also the $k_{\gamma_i}$ factor that appears in $\theta_{\gamma_i}$ which are optimized, and which would give us, subject to satisfying \cref{eq81,eq82}, the open intervals
\begin{align}
   I_1= \Bigg[ \Bigg( \prod\limits_{i\in A} \Bigg( 1-\bigg(\frac{1-\alpha_i}{2}\bigg)^{k_{\alpha_i}} \Bigg) \Bigg) \Bigg(  1-\bigg(\frac{1-\beta_1}{2}\bigg)^{k_{\beta_1}} \Bigg) , \Bigg( \prod\limits_{i\in A} \Bigg( 1-\bigg(\frac{1-\alpha_i}{2}\bigg)^{k_{\alpha_i}} \Bigg) \Bigg)\Bigg],
\end{align}

\begin{align}
   I_2= \Bigg[ \Bigg( \prod\limits_{i=1}^{j} \Bigg( 1-\bigg(\frac{1-\beta_i}{2}\bigg)^{k_{\beta_i}} \Bigg) \Bigg), \Bigg( \prod\limits_{i=2}^{j} \Bigg( 1-\bigg(\frac{1-\beta_i}{2}\bigg)^{k_{\beta_i}} \Bigg) \Bigg)\Bigg]
\end{align}

Further, subject to satisfying the constraints of \cref{eq83,eq84}, we get the following two intervals: 
\begin{align}
   I_3= \Bigg[ \Bigg( \prod\limits_{i\in B} \Bigg( 1-\bigg(\frac{1-\beta_i}{2}\bigg)^{k_{\beta_i}} \Bigg) \Bigg) \Bigg(  1-\bigg(\frac{1-\alpha_1}{2}\bigg)^{k_{\alpha_1}} \Bigg) , \Bigg( \prod\limits_{i\in B} \Bigg( 1-\bigg(\frac{1-\beta_i}{2}\bigg)^{k_{\beta_i}} \Bigg) \Bigg)\Bigg],
\end{align}

\begin{align}
   I_4= \Bigg[ \Bigg( \prod\limits_{i=1}^{k} \Bigg( 1-\bigg(\frac{1-\alpha_i}{2}\bigg)^{k_{\alpha_i}} \Bigg) \Bigg), \Bigg( \prod\limits_{i=2}^{k} \Bigg( 1-\bigg(\frac{1-\alpha_i}{2}\bigg)^{k_{\alpha_i}} \Bigg) \Bigg)\Bigg]
\end{align}

Finally, we note that the intervals obtained here will change if some of the Cantor sets have support in a unit interval other than $[0,1]$. Further, one can employ the methods of Section 4 to expand the open intervals obtained in the case of products of Cantor sets, which we have not pursued here. 

\section{Sums and products under images of $C^1$ maps.}

We conclude with an outline of the results one would get for one function $\phi(x)$ that is $C^1$ and supported in $[0,1]$. \footnote{One can in principle extend these to multiple functions with different Cantor sets, which we don't do here.} 

For the problem of the sumsets, suppose that the first derivative being continuous and thus bounded on $[\frac{1+\alpha}{2},1]$, is such that $0\leq g_2\leq \phi'(x)\leq g_1$  and $\phi$ is monotonic on $[\frac{1+\alpha}{2},1]$. In this case, one can use the mean value theorem and these bounds to alter the upper bound on \cref{eqimp1} and the upper and lower bounds on \cref{eqimp2}, so that we would get the following.

Fix $\alpha \in (0,1)$. For this fixed $\alpha$ let $C_{\alpha}$ be the Cantor set as in the introduction. Also, let
$$s_{\alpha,\phi} = \left\lceil \frac{g_1}{g_2}\right\rceil, \qquad k_{\alpha,\phi}=\begin{cases} \Bigg\lceil \frac{\frac{3\alpha -1}{2}}{\big(\frac{g_2}{g_1}\big)\big(\frac{1-\alpha^2}{4}\big)} \Bigg\rceil  \\
0 & \alpha \leq 1/3 \end{cases},$$
\noindent and let $r_{\alpha,\phi}=2s_{\alpha,\phi}+2k_{\alpha,\phi}$. Set  $$\Gamma_{\alpha, \phi} = \left\{\sum_{k=1}^{r_{\alpha,\phi}} \phi(c_k) : c_k \in C_{\alpha}\right\}$$

Then one has the interval $I$ similar to the one defined in Theorem 2;

$$I = \Bigg[\Big(\frac{r_{\alpha,m}}{2}-1\Big)\phi(1)+\Big(\frac{r_{\alpha,m}}{2}+1\Big)\phi\Big(\frac{1+\alpha}{2}\Big), \ \Big(\frac{r_{\alpha,m}}{2}+1\Big)\phi(1)+\Big(\frac{r_{\alpha,m}}{2}-1\Big)\phi\Big(\frac{1+\alpha}{2}\Big)\Bigg] $$ and that we have $I\subset \Gamma_{\alpha, \phi} := \Bigg\{\sum\limits_{k=1}^{r_{\alpha,\phi}} \phi(c_k) : c_k \in C_{\alpha}\Bigg\} $.

\bigskip

Consider the case of products, and the quantities,

\begin{align}\label{eq:eqprod}
    s_\alpha=\Bigg\lceil \frac{\frac{g_3}{g_4}}{ \Bigg(\phi\Big(1-\Big(\frac{1-\alpha}{2}\Big)^{k_\alpha}\Big) \Bigg)^{2t_\alpha -1}} \Bigg\rceil, \  p_{\alpha}=\begin{cases} \Bigg\lceil \frac{\frac{g_3}{g_4}(\frac{3\alpha -1}{2})}{\Big(\phi\Big( 1- \Big(\frac{1-\alpha}{2}\Big)^{k_\alpha} \Big)\Big)^{2t_\alpha -1}(\frac{1-\alpha^2}{4})} \Bigg\rceil & \alpha>\frac{1}{3} \\ 0 & \alpha\leq \frac{1}{3} \end{cases},
\end{align}

where we have the bounds $0\leq g_4\leq \phi'(x)\leq g_3$ and $\phi$ is monotonic on the range $\Big[\Big(1-\Big(\frac{1-\alpha}{2}\Big)^{k_\alpha},1\Big]$.

Then we contain the interval

\begin{align}
    I_\alpha=\Bigg[\Bigg(\phi\Big(1-\Big(\frac{1-\alpha}{2}\Big)^{k_\alpha}\Big)\Bigg)^{t_\alpha+1}, \Bigg(\phi\Big(1-\Big(\frac{1-\alpha}{2}\Big)^{k_\alpha}\Big)\Bigg)^{t_\alpha -1}\Bigg]\subset \Gamma_{\alpha, \phi} := \Bigg\{\prod\limits_{k=1}^{2 t_{\alpha}} \phi(c_k) : c_k \in C_{\alpha}\Bigg\}.
\end{align}

Here, as in Section 6, we again have 

$$\begin{cases}
    t_\alpha=s_{\alpha}  &   \alpha\leq\frac{1}{3}\\

    t_\alpha=s_{\alpha} +p_\alpha  &  \alpha>\frac{1}{3}.
\end{cases}$$

Note that if the function $\phi$ is positive within $[0,1]$ one can alter the argument to run the dynamics within the entire interval $[0,1]$ as opposed to the restricted interval $[\frac{1+\alpha}{2},1]$.

\section{Acknowledgements} The author is grateful to Arun Suresh for many discussions, and to Pablo Shmerkin, Anton Gorodetski and Petros Valettas for feedback on these questions.

\end{document}